\RequirePackage{etoolbox}
\csdef{input@path}{{style/}{graphics/}}
\documentclass[aap,MSNbibl,dvips]{arximspdf}
\makeatletter
   \@ifpackageloaded{graphicx}{}{\usepackage{graphicx}}
\makeatother

\doi{10.1214/14-AAP1018} 
\volume{25}
\issue{3}
\pubyear{2015}
\firstpage{1108}
\lastpage{1154}

\makeatletter

\newcommand{\eqref}[1]{(\ref{#1})}
\newtheorem{theor}{Theorem}
\newproclaim{defin}[theor]{Definition}

\newtheorem{lemma}[theor]{Lemma}

\newcommand{\Z}{\mathbb{Z}}
\newcommand{\R}{\mathbb{R}}
\newcommand{\D}{\mathbb{D}}
\newcommand{\ind}{\mathbf{1}}
\newcommand{\ep}{\varepsilon}

\newcommand{\card}{\operatorname{card}}
\newcommand{\dist}{\operatorname{dist}}
\newcommand{\poisson}{\operatorname{Poisson}}
\newcommand{\exponential}{\operatorname{Exponential}}
\makeatother

\begin{document}
\begin{frontmatter}

\title{Evolutionary games on the lattice: Payoffs affecting birth and death rates}
\runtitle{Evolutionary games on the lattice}

\begin{aug}
\author[A]{\fnms{N.}~\snm{Lanchier}\corref{}\thanksref{T1}\ead[label=e1]{nicolas.lanchier@asu.edu}}
\thankstext{T1}{Supported in part by NSF Grant DMS-10-05282.}
\runauthor{N. Lanchier}
\affiliation{Arizona State University}
\address[A]{School of Mathematical and Statistical Sciences\\
Arizona State University\\
Tempe, Arizona 85287\\
USA\\
\printead{e1}} 
\end{aug}

\received{\smonth{12} \syear{2013}}

%
\begin{abstract}
This article investigates an evolutionary game based on the framework
of interacting particle systems.
Each point of the square lattice is occupied by a player who is
characterized by one of two possible strategies and is attributed
a payoff based on her strategy, the strategy of her neighbors and a
payoff matrix.
Following the traditional approach of evolutionary game theory, this
payoff is interpreted as a fitness: the dynamics of the
system is derived by thinking of positive payoffs as birth rates and
the absolute value of negative payoffs as death rates.
The nonspatial mean-field approximation obtained under the assumption
that the population is well mixing is the popular replicator
equation.
The main objective is to understand the consequences of the inclusion
of local interactions by investigating and comparing the
phase diagrams of the spatial and nonspatial models in the four
dimensional space of the payoff matrices.
Our results indicate that the inclusion of local interactions induces
a reduction of the coexistence region of the replicator
equation and the presence of a dominant strategy that wins even when
starting at arbitrarily low density in the region where
the replicator equation displays bistability.
We also discuss the implications of these results in the parameter
regions that correspond to the most popular games:
the prisoner's dilemma, the stag hunt game, the hawk-dove game and the
battle of the sexes.
\end{abstract}

%

\end{frontmatter}

\section{Introduction}
\label{sec:intro}

The book of von Neumann and Morgenstern \cite
{vonneumann_morgenstern_1944} that develops mathematical methods to understand
human behavior in strategic and economic decisions is the first
foundational work in the field of game theory.
The most popular games are symmetric two-person games whose
characteristics are specified by a square matrix where the common number
of rows and columns denotes the number of possible pure strategies,
and the coefficients represent the player's payoffs which depend
on both her strategy and the strategy of her opponent.
Game theory relies on the assumption that players are rational decision-makers.
In particular, the main question in this field is:
what is the best possible response against a player who tries to
maximize her payoff?
The work of Nash \cite{nash_1950} on the existence of Nash
equilibrium, a mathematical criterion for mutual consistency of players'
strategies, is an important contribution that gives a partial answer
to this question.

In contrast, the field of evolutionary game theory, which was
proposed by theoretical biologist Maynard Smith and first
appeared in his work with Price \cite{maynardsmith_price_1973}, does
not assume that players make rational decisions:
evolutionary game theory makes use of concepts from traditional game
theory to describe the dynamics of populations by
thinking of individuals as interacting players and their trait as a
strategy, and by interpreting their payoff as a fitness or
reproduction success.
The analog of Nash equilibrium in evolutionary game theory is called
ESS, a short for evolutionary stable strategy, and is defined
as a strategy which, if adopted by a population, cannot be invaded by
any alternative strategy starting at an infinitesimally small
frequency.
This key concept first appeared in the foundational work~\cite
{maynardsmith_price_1973}.
Even though evolutionary games were originally introduced to
understand the outcome of animal conflicts, they now have a wide
variety of applications as a powerful framework to study interacting
populations in which the reproductive success of each individual
is density dependent, a key component of social and biological communities.

The inclusion of stochasticity and space in the form of local
interactions is another key factor in how communities are
shaped, and evolutionary games have been studied through both the
mathematical analysis of deterministic nonspatial models based on
differential equations and simulations of more complex models based on
spatial stochastic processes.
For a review on deterministic nonspatial evolutionary games, we refer
to \cite{hofbauer_sigmund_1998}.
On the side of spatial stochastic evolutionary games, one important
contribution is the work of Nowak and
May~\cite{nowak_may_1992,nowak_may_1993} which, based on simulations
of cellular automata, shows that space favors cooperation
in the prisoner's dilemma.
See also~\cite{nowak_bonhoeffer_may_1994a,nowak_bonhoeffer_may_1994b}
for similar works based on asynchronous updating models,
\cite{nowak_2006} for a review on spatial evolutionary games and \cite
{fu_nowak_hauert_2010,langer_nowak_hauert_2010} and
references therein for more recent work on the topic.
The rigorous analysis of nonspatial deterministic models and
simulations of spatial stochastic models are both important and
complementary but also have some limitations:
spatial simulations suggest that nonspatial models fail to
appropriately describe systems, including local interactions, but are
known at the same time to be difficult to interpret, leading sometimes
to erroneous conclusions.
This underlines the necessity of an analytical study of evolutionary
games based on stochastic spatial models.
References \cite{chen_2013,cox_durrett_perkins_2013,durrett_2013}
are, as far as we know, the only three articles that also carry
out a rigorous analysis of such models but the authors' approach
significantly differs from ours: they assume that
\[
\mbox{fitness} = (1 - w) + w \times\mbox{payoff}\quad \mbox{and} \quad w \to0,
\]
which is referred to as weak selection and allows for a complete
analytical treatment using voter model perturbation techniques.
Indeed, for $w = 0$, their model reduces to the popular voter model
introduced in \cite{clifford_sudbury_1973,holley_liggett_1975}.
In contrast, we assume that fitness${} = {}$payoff, which makes our model
mathematically more challenging and does not allow
for a complete analysis.
However, the limiting behavior in different parameter regions can be
understood based on various techniques, which leads to
interesting findings about the consequences of the inclusion of local
interactions.
For a similar approach, we also refer to \cite
{evilsizor_lanchier_2014} where the best-response dynamics, a~spatial
process in
which players update their strategy at a constant rate in order to
maximize their payoff, is studied.


\subsection*{The replicator equation}

As previously mentioned, most of the analytical works in
evolutionary game theory are based on ordinary differential equations.
The most popular model that falls into this category is the replicator
equation, which we describe for simplicity in the presence of
only two strategies since this is the case under consideration for the
stochastic spatial model we introduce later.
The dynamics depends on a $2 \times2$ payoff matrix $A = (a_{ij})$
where $a_{ij}$ denotes the payoff of a player who follows
strategy~$i$ interacting with a player who follows strategy~$j$.
To formulate the replicator equation and describe its bifurcation
diagram, it is convenient to use the terminology introduced
by the author in \cite{lanchier_2012} by setting
\[
a_1:= a_{11} - a_{21} \quad\mbox{and}\quad
a_2:= a_{22} - a_{12}
\]
and declaring strategy $i$ to be:
\begin{itemize}
\item\emph{altruistic} when $a_i < 0$, that is, a player with strategy
$i$ confers a lower payoff to a player following the same strategy
than to a player following the other strategy,
\item\emph{selfish} when $a_i > 0$, that is, a player with strategy $i$
confers a higher payoff to a player following the same strategy
than to a player following the other strategy.
\end{itemize}
The replicator equation is a system of coupled differential equations
for the frequency $u_i$ of players following strategy $i$.
The payoff of each type $i$ player is given by
%
\begin{equation}
\label{eq:payoff-nospace} \phi_i = \phi_i (u_1,
u_2):= a_{i1} u_1 + a_{i2}
u_2 \qquad\mbox{for } i = 1, 2.
\end{equation}
Interpreting this payoff as the growth rate of each type $i$ player,
using that the frequencies sum up to one, and recalling
the definition of $a_1$ and $a_2$, one obtains the following ordinary
differential equation, the so-called replicator equation,
for the frequency of type~1 players:
%
\begin{eqnarray}
\label{eq:replicator} %
 u_1' (t)
& = & u_1 u_2 (\phi_1 - \phi_2)
= u_1 u_2 (a_{11} u_1 +
a_{12} u_2 - a_{21} u_1 -
a_{22} u_2)
\nonumber
\\[-8pt]
\\[-8pt]
\nonumber
& = & u_1 u_2 (a_1 u_1 -
a_2 u_2) = u_1 (1 - u_1)
\bigl((a_1 + a_2) u_1 - a_2
\bigr).
\end{eqnarray}
The system has three fixed points, namely
\[
e_1:= 1 \quad\mbox{and} \quad e_2:= 0\quad \mbox{and}\quad
e_{12}:= a_2 (a_1 + a_2)^{-1},
\]
and basic algebra shows that the limiting behavior only depends on the
sign of $a_1$ and $a_2$, therefore on whether strategies are
altruistic or selfish.
More precisely, we find that:
\begin{itemize}
\item when strategy~1 is selfish and strategy~2 altruistic, strategy~1 wins:
$e_{12} \notin(0, 1)$ and starting from any initial condition $u_1
(0) \in(0, 1)$, $u_1 \to e_1$.
\item when strategy~1 is altruistic and strategy~2 selfish, strategy~2 wins:
$e_{12} \notin(0, 1)$ and starting from any initial condition $u_1
(0) \in(0, 1)$, $u_1 \to e_2$.
\item when both strategies are altruistic, coexistence occurs:
$e_{12} \in(0, 1)$ is globally stable, that is, starting from any
initial condition $u_1 (0) \in(0, 1)$, $u_1 \to e_{12}$.
\item when both strategies are selfish, the system is bistable:
$e_{12} \in(0, 1)$ is unstable, and $u_1$ converges to either $e_1$
or $e_2$ depending on whether it is initially larger or smaller than $e_{12}$.
\end{itemize}
In terms of evolutionary stable strategy, this indicates that, for
well-mixed populations, a strategy is evolutionary stable if it
is selfish, but not if it is altruistic.


\subsection*{Spatial analog}

To define a spatial analog of the replicator equation, we
employ the framework of interacting particle systems by positioning
the players on an infinite grid.
Our spatial game is then described by a continuous-time Markov chain
$\eta_t$ whose state space maps the $d$-dimensional lattice into
the set of strategies $\{1, 2 \}$, with $\eta_t (x)$ denoting the
strategy at vertex $x$.
Players being located on a geometrical structure, space can be
included in the form of local interactions by assuming that the
payoff of each player is computed based on the strategy of her neighbors.
More precisely, we define the interaction neighborhood of vertex $x$ as
\[
N_x:= \Bigl\{y \in\Z^d
\dvtx y \neq x \mbox{ and } \max_{i = 1, 2, \ldots, d} |y_i -
x_i| \leq M \Bigr\}\qquad \mbox{for } x \in\Z^d,
\]
where $M$ is referred to as the dispersal range.
Letting $f_j (x, \eta)$ denote the fraction of type~$j$ players in the
neighborhood of vertex $x$, the payoff of $x$ is then defined as
\[
\phi\bigl(x, \eta| \eta(x) = i\bigr):= a_{i1} f_1 (x,
\eta) + a_{i2} f_2 (x, \eta) \qquad\mbox{for } i = 1, 2,
\]
which can be viewed as the spatial analog of \eqref{eq:payoff-nospace}.
The dynamics is again derived by interpreting the payoff as a fitness.
More precisely, we think of a payoff as either a birth rate or a death
rate depending on its sign:
if the player at vertex $x$ has a positive payoff, then at rate this
payoff, one of her neighbors chosen uniformly at random
adopts the strategy at $x$, while if she has a negative payoff, then
at rate minus this payoff, she adopts the strategy of one
of her neighbors again chosen uniformly at random.
This is described formally by the Markov generator
%
\begin{eqnarray}
\label{eq:brown-hansell} %
&&Lf (\eta) \nonumber\\
&&\qquad= N^{-1} \sum
_x \sum_{y \in N_x} \phi(y,
\eta) \ind\bigl\{\phi(y, \eta) > 0 \bigr\} \ind\bigl\{\eta(x) \neq\eta(y) \bigr\}
\bigl[f (\eta_x) - f (\eta )\bigr]
\\
&&\qquad\quad{}- N^{-1} \sum
_x \sum_{y \in N_x} \phi(x, \eta)
\ind\bigl\{\phi(x, \eta) < 0 \bigr\} \ind\bigl\{\eta(x) \neq\eta(y) \bigr\}
\bigl[f (\eta_x) - f (\eta )\bigr], \nonumber
\end{eqnarray}
where configuration $\eta_x$ is obtained from configuration $\eta$ by
changing the strategy at vertex $x$ and leaving the strategy
at the other vertices unchanged, and where the constant $N$ is simply
the common size of the interaction neighborhoods.
Note that $L$ indeed defines a unique Markov process according to, for
example, Theorem B3 in Liggett \cite{liggett_1999}.
Model \eqref{eq:brown-hansell} is inspired from the spatial version of
Maynard Smith's evolutionary games introduced by Brown and
Hansell \cite{brown_hansell_1987}.
Their model allows for any number of players per vertex, and the
dynamics includes three components: migration, death due to crowding
and birth or death based on the value of the payoff.
Our model only retains the third component.
We also point out that the model obtained from~\eqref
{eq:brown-hansell} by assuming that the population is well
mixing, called the mean-field approximation, is precisely the
replicator equation~\eqref{eq:replicator}. Therefore the consequences
of the inclusion of space and stochasticity can indeed be understood
through the comparison of both models, which as we show later,
disagree in many ways.
In fact, the original model of Brown and Hansell is also studied
numerically in \cite{durrett_levin_1994} where it is used to argue
that the inclusion of space and/or stochasticity can lead to drastic
behavior changes.


\subsection*{Main results}

We now study the limiting behavior of the spatial game.
Unless explicitly stated otherwise, all the results below apply to the
process starting from Bernoulli product
measures in which the density of each strategy is positive and
constant across space:
%
\begin{equation}
\label{eq:product} P \bigl(\eta_0 (x_1) =
\eta_0 (x_2) = \cdots= \eta_0
(x_n) = 1\bigr) = \rho ^n \qquad\mbox{for some }  \rho\in(0, 1)
\end{equation}
and every finite sequence $x_1, x_2, \ldots, x_n$ of distinct vertices.
From the point of view of the replicator equation, whether a strategy
wins, or both strategies coexist or the system is bistable
is defined based on the value of the nontrivial fixed point $e_{12}$
and the stability of this and the other two fixed points.
For the spatial game, we say that:
\begin{itemize}
\item strategy $i \in\{1, 2 \}$ \emph{survives} whenever
\[
P \bigl(\eta_s (x) = i \mbox{ for some } s > t\bigr) = 1 \qquad\mbox{for
all } (x, t) \in\Z^d \times\R_+;
\]
\item strategy $i \in\{1, 2 \}$ \emph{goes extinct} whenever
\[
\lim_{t \to\infty} P \bigl(
\eta_t (x) = i\bigr) = 0 \qquad\mbox{for all } x \in\Z^d;
\]
\item a strategy \emph{wins} if it survives whereas the other strategy
goes extinct;
\item both strategies \emph{coexist} whenever
\[
\liminf_{t \to\infty} P \bigl(
\eta_t (x) \neq\eta_t (y)\bigr) > 0 \qquad\mbox{for all } x, y
\in\Z^d;
\]
\item the system \emph{clusters} whenever
\[
\lim_{t \to\infty} P \bigl(
\eta_t (x) \neq\eta_t (y)\bigr) = 0 \qquad\mbox{for all } x, y
\in\Z^d.
\]
\end{itemize}
Numerical simulations suggest that in the presence of one selfish and
one altruistic strategy, the selfish strategy wins, just
as in the replicator equation.
In contrast, when both strategies are selfish, spatial and nonspatial
models disagree.
Our brief analysis of the replicator equation indicates that the
system is bistable: both strategies are ESS.
The transition curve for the spatial model is difficult to find based
on simulations, but simple heuristic arguments looking at the
interface between two adjacent blocks of the two strategies suggest
that the most selfish strategy, that is, the one with the largest $a_i$,
always wins even when starting at a very low density, thus indicating
that only the most selfish strategy is an ESS.
The fact that bistability in the mean-field model results in the
presence of a strong type in the interacting particle system has already
been observed for a number of models, and we refer to \cite
{durrett_2009,durrett_levin_1994} for such examples.
For two altruistic strategies, coexistence is again possible, but the
coexistence region of the spatial game is significantly smaller
than that of the replicator equation: except in the one-dimensional
nearest neighbor case, coexistence occurs in a thorn-shaped region
starting at the bifurcation point $a_1 = a_2 = 0$.
The smaller the range of the interactions and the spatial dimension,
the smaller the coexistence region.
In the one-dimensional nearest neighbor case, the simulations are
particularly difficult to interpret when
%
\begin{equation}
\label{eq:clustering} a_{11} + a_{12} < 0 < a_{12}
\quad\mbox{and}\quad a_{22} + a_{21} < 0 < a_{21}.
\end{equation}
See Figure~\ref{fig:cluster} for a picture of two realizations when
\eqref{eq:clustering} holds.
However, we were able to prove that the one-dimensional nearest
neighbor system clusters except in a parameter region with measure zero in
the space of the $2 \times2$ matrices in which all the players have a
zero payoff eventually, thus leading to a fixation of the system in a
configuration in which both strategies are present.
More generally, we conjecture that, except in this parameter region
with measure zero, the least altruistic strategy always wins, just
as in the presence of selfish-selfish interactions.
The thick continuous lines on the right-hand side of Figures~\ref
{fig:diagrams-1} and \ref{fig:diagrams-2} summarize our conjectures
for the spatial game in the one-dimensional nearest neighbor case and
all the other cases, respectively.
These results are reminiscent of the ones obtained for the models introduced
in \cite{lanchier_2012,lanchier_neuhauser_2009,neuhauser_pacala_1999}, which though they are not examples of
evolutionary games,
also include density-dependent birth or death rates.
Our proofs and the proofs in these three references strongly differ
while showing the same pattern: for all four models,
the inclusion of local interactions induces a reduction of the
coexistence region of the mean-field model, and there is a dominant type
that wins even when starting at arbitrarily low density in the region
where the mean-field model displays bistability.

We now state our analytical results for the spatial stochastic
process, which confirms in particular these two important aspects.
To motivate and explain our first result, we observe that the presence
of density-dependent birth and death rates typically precludes the
existence of a mathematically tractable dual process.
See Liggett \cite{liggett_1985}, Section II.3, for a general
definition of duality and dual process.
Note, however, that if $a_{11} = a_{12}$ and $a_{22} = a_{21}$, then
the payoff of players of either type is constant across all possible
spatial configurations: birth and death rates are no longer density dependent.
For this specific choice of the payoffs, the process reduces to a
biased voter model~\cite{bramson_griffeath_1980,bramson_griffeath_1981},
and therefore strategy~1 wins if in addition $a_{12} > a_{21}$.
For all other payoff matrices, the dynamics is more complicated, but
there is a large parameter region in which the spatial game can still
be coupled with a biased voter model to deduce that strategy~1 wins.
This parameter region is specified in the following theorem.
See also the right-hand side of Figure~\ref{fig:diagrams-2} where the
boundary of this region is represented in dashed lines on the
$a_{11} - a_{22}$ plane.

%
\begin{theor}
\label{th:biased}
Assume that $a_{12} > a_{21}$.
Then strategy~1 wins whenever
\[
\max(a_{22}, a_{21}) + a_{21} (N -
1)^{-1} < \min (a_{11}, a_{12}) + a_{12}
(N - 1)^{-1}.
\]
\end{theor}

This parameter region intersects the regions in which the replicator
equation displays coexistence and bistability.
In particular, the theorem confirms that the inclusion of local
interactions induces a reduction of the coexistence and bistable regions
in accordance with the simulation results mentioned above.
The next two results strengthen this theorem by proving that the
parameter region in which there is a unique~ESS extends to arbitrarily
small/large values of $a_{11}$ and $a_{22}$.
To state these results, it is convenient to introduce the vector $\bar
a_{11}:= (a_{12}, a_{21}, a_{22})$.

\begin{theor}
\label{th:richardson}
For all $\bar a_{11}$ there exists $m (\bar a_{11}) < \infty$ such that
\[
\mbox{strategy 1 wins whenever } a_{11} > m (\bar a_{11}).
\]
\end{theor}

This implies that, in accordance with our numerical simulations, the
parameter region in which the replicator equation is
bistable while there is a unique ESS for the spatial game is much
larger than the parameter region covered by Theorem~\ref{th:biased}.
Note also that, in view of the symmetry of the model, the previous
theorem also holds by exchanging the roles of the two strategies.

\begin{theor}
\label{th:walks}
For all $\bar a_{11}$ such that $a_{12} < 0$, there exists $m (\bar
a_{11}) < \infty$ such that
\[
\mbox{strategy 2 wins whenever } a_{11} < - m (\bar a_{11}).
\]
\end{theor}

This implies that, again in accordance with our numerical simulations,
the parameter region in which coexistence occurs for the
replicator equation while there is a unique ESS for the spatial game
is much larger than the one covered by Theorem~\ref{th:biased}.
Once more, we point out that, in view of the symmetry of the model,
the theorem also holds by exchanging the roles of the two strategies
as indicated in Figure~\ref{fig:diagrams-2}.
The previous two results hold regardless of the spatial dimension and
the range of the interactions and can be significantly improved
in the one-dimensional nearest neighbor case through an analysis of
the boundaries of the system.
More precisely, letting
%
\begin{eqnarray}
\label{eq:measure-zero} %
\mathcal M_2^*&:=& \bigl
\{A = (a_{ij}) \mbox{ such that}
\nonumber
\\[-8pt]
\\[-8pt]
\nonumber
&&\hspace*{6pt} a_{11} a_{12} a_{21} a_{22}
(a_{11} + a_{12}) (a_{22} + a_{21}) \neq0
\mbox{ and } a_{11} + a_{12} \neq a_{22} +
a_{21} \bigr\}\hspace*{-10pt}
\end{eqnarray}
we have the following theorem.

%
\begin{theor}
\label{th:1D}
Assume that $M = d = 1$. Then:
\begin{itemize}
\item strategy~1 wins for all $a_{11} > \max(a_{22}, a_{21}) +
(a_{21} - a_{12})$, and
\item the system starting from any translation invariant distribution
clusters for all $A \in\mathcal M_2^*$.
\end{itemize}
\end{theor}

%
\begin{figure}

\includegraphics{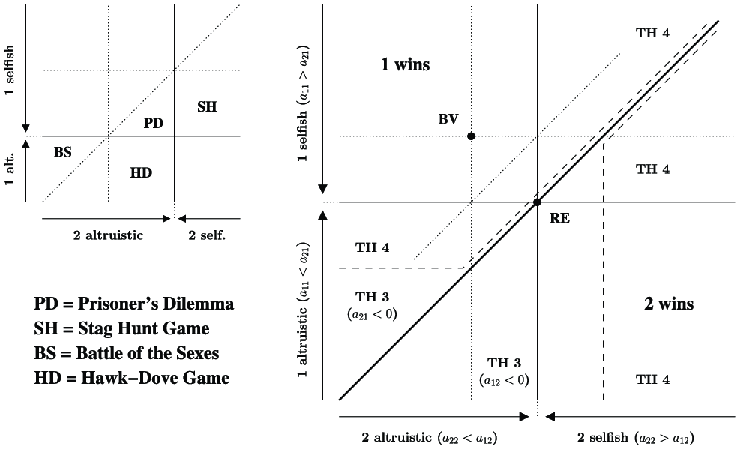}

\caption{List of the most popular $2 \times2$ games on the
left and phase diagrams of the spatial game
along with a summary of the theorems in the $a_{11} - a_{22}$ plane on
the right.
The thick lines refer to the transition curves suggested by the simulations.
BV${} = {}$biased voter model, and RE${} = {}$replicator equation.}
\label{fig:diagrams-1}
\end{figure}

Figure~\ref{fig:diagrams-1} gives the phase diagram of the
one-dimensional nearest neighbor process obtained by
combining Theorems \ref{th:walks} and~\ref{th:1D}.
The parameter region in the first part of Theorem~\ref{th:1D}, and the
one obtained by symmetry are represented in dashed lines
when $a_{12} > a_{21}$.
For most of the parameter region in which at least one strategy is
selfish, the most selfish strategy wins, which significantly
improves Theorem~\ref{th:richardson}.
The second part of the theorem supplements Theorem~\ref{th:walks} by
proving that, except in the measure zero parameter region that
corresponds to $A \in\mathcal M_2^*$, coexistence is not possible in
the one-dimensional nearest neighbor case.
Finally, our last theorem looks more closely at the interactions
between two altruistic strategies and confirms that, except in the
one-dimensional nearest neighbor case, coexistence is possible for the
spatial game.

\begin{theor}
\label{th:coex}
There is $m:= m (a_{12}, a_{21}) > 0$ such that coexistence occurs when
\[
c (M, d) a_{22} < a_{11} < - m\quad \mbox{and}\quad c (M, d)
a_{11} < a_{22} < - m,
\]
where, for each range $M$ and spatial dimension $d$,
\[
c (M, d):= \frac{2M ((2M + 1)^d - 2)}{(M + 1) (2M (2M +
1)^{d - 1} - 1)}.
\]
\end{theor}

Note that the parameter region in the theorem is nonempty if and only
if $c (M, d)$ is strictly larger than one.
In addition, the larger $c (M, d)$, the larger this parameter region.
This, together with Table~\ref{tab:drift}, indicates that, except in
the one-dimensional nearest neighbor case in which the region
given in the theorem is empty, the coexistence region contains an
infinite subset of a certain triangle whose range increases with
both the dispersal range and the spatial dimension.
We refer to Figure~\ref{fig:diagrams-2} for a summary of the theorems
that exclude the one-dimensional nearest neighbor case.

%
\begin{table}[t]
\tabcolsep=0pt
\caption{$c (M, d)$ for different values of the range $M$ and
the dimension $d$}
\label{tab:drift}
\begin{tabular*}{\textwidth}{@{\extracolsep{\fill}}lccccccccc@{}}
\hline
& \multicolumn{1}{c}{$\bolds{d = 1}$} & \multicolumn{1}{c}{$\bolds{d = 2}$} &
\multicolumn{1}{c}{$\bolds{d = 3}$} & \multicolumn{1}{c}{$\bolds{d = 4}$} &
\multicolumn{1}{c}{$\bolds{d = 5}$} & \multicolumn{1}{c}{$\bolds{d = 6}$} &
\multicolumn{1}{c}{$\bolds{d = 7}$}
& \multicolumn{1}{c}{$\bolds{d = 8}$} & \multicolumn{1}{c@{}}{$\bolds{d = 9}$}
 \\
\hline
$M = 1$ & 1.0000 & 1.4000 & 1.4706 & 1.4906 & 1.4969 & 1.4990 & 1.4997
& 1.4999 & 1.5000 \\
$M = 2$ & 1.3333 & 1.6140 & 1.6566 & 1.6647 & 1.6663 & 1.6666 & 1.6667
& 1.6667 & 1.6667 \\
$M = 3$ & 1.5000 & 1.7195 & 1.7457 & 1.7494 & 1.7499 & 1.7500 & 1.7500
& 1.7500 & 1.7500 \\
$M = 4$ & 1.6000 & 1.7803 & 1.7978 & 1.7998 & 1.8000 & 1.8000 & 1.8000
& 1.8000 & 1.8000 \\
$M = 5$ & 1.6667 & 1.8196 & 1.8321 & 1.8332 & 1.8333 & 1.8333 & 1.8333
& 1.8333 & 1.8333 \\
$M = 6$ & 1.7143 & 1.8470 & 1.8564 & 1.8571 & 1.8571 & 1.8571 & 1.8571
& 1.8571 & 1.8571 \\
$M = 7$ & 1.7500 & 1.8672 & 1.8745 & 1.8750 & 1.8750 & 1.8750 & 1.8750
& 1.8750 & 1.8750 \\
$M = 8$ & 1.7778 & 1.8827 & 1.8885 & 1.8889 & 1.8889 & 1.8889 & 1.8889
& 1.8889 & 1.8889 \\
$M = 9$ & 1.8000 & 1.8950 & 1.8997 & 1.9000 & 1.9000 & 1.9000 & 1.9000
& 1.9000 & 1.9000\\
\hline
\end{tabular*}
\end{table}

\begin{figure}[b]

\includegraphics{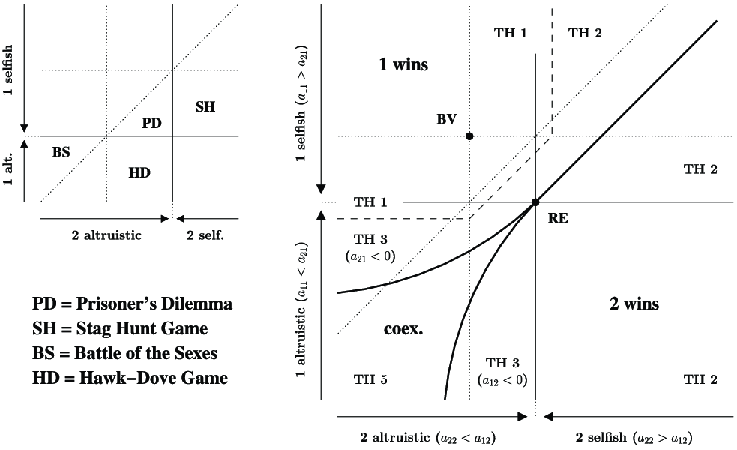}

\caption{List of the most popular $2 \times2$ games on the
left and phase diagrams of the spatial game
along with a summary of the theorems in the $a_{11} - a_{22}$ plane on
the right.
The thick lines refer to the transition curves suggested by the simulations.
BV${} = {}$biased voter model, and RE${} = {}$replicator equation.}
\label{fig:diagrams-2}
\end{figure}


\subsection*{The role of space in the most popular games}

The last step before going into the details of the proofs is to
discuss the implications of our results in the most
popular symmetric two-person games.
To define these games, note that there are $4! = 24$ possible
orderings of the four payoffs therefore, also accounting for
symmetry, twelve possible strategic situations corresponding to twelve
symmetric two-person games involving two strategies.
These twelve regions of the parameter space are represented in the
$a_{11} - a_{22}$ plane on the left-hand diagrams of
Figures~\ref{fig:diagrams-1}--\ref{fig:diagrams-2} along with the
names of the most popular games under the assumption
$a_{12} > a_{21}$.

\emph{Prisoner's dilemma.}
The prisoner's dilemma is probably the most popular symmetric
two-person game.
When $a_{12} > a_{21}$, strategy~1 means defection whereas strategy~2
means cooperation.
From the point of view of the replicator equation, defection is the
only ESS.
Numerical simulations suggest that the same holds for our spatial
model, which is covered in part in the general case and completely
in the one-dimensional nearest neighbor case in Theorems~\ref
{th:biased} and~\ref{th:1D}.

\emph{Stag hunt.}
In the stag hunt game with $a_{12} > a_{21}$, strategy~1 represents
safety: hunting a hare, whereas strategy~2 represents social
cooperation: hunting a stag.
In the absence of space, both strategies are evolutionary stable.
In contrast, Theorem~\ref{th:richardson} shows that, in the presence
of local interactions, social cooperation is the only ESS
if the reward $a_{22}$ for social cooperation is high enough, that is,
a stag is worth much more than a hare, whereas if the reward
is not significant, then safety becomes the only ESS according to
Theorems~\ref{th:biased} and~\ref{th:1D}.

\emph{Hawk-dove.}
In the hawk-dove game, strategy~1 represents hawks that fight for the
resource and strategy~2, doves that share peacefully the resource.
The cost of a fight is larger than the value of the contested
resource, which makes this game an example of anti-coordination game:
the best possible response to a strategy is to play the other strategy.
In the absence of space, none of the strategies is evolutionary stable
so coexistence occurs.
In contrast, Theorem~\ref{th:walks} indicates that, in the presence of
a spatial structure, the dove strategy is the only ESS
when the cost of an escalated fight is high enough or equivalently
when $a_{11}$ is small enough.

\emph{Battle of the sexes.}
In the battle of the sexes, husband and wife cannot remember if they
planned to meet at the opera or at the football match.
The husband would prefer the match and the wife the opera, but overall
both would prefer to go to the same place.
Mutual cooperation, that is, both go to the place that the other
prefers, leads to the lowest possible payoff $a_{22}$ which makes this
game another example of anti-coordination game.
In the absence of space, none of the strategies is evolutionary
stable, so coexistence occurs.
Coexistence is also possible in the presence of a spatial structure
according to Theorem~\ref{th:coex}.
However, if the cost of mutual cooperation is too high, that is,
$a_{22}$ too small, then defection becomes the unique ESS according
to Theorem~\ref{th:walks}.

The rest of the paper is devoted to the proofs of the theorems.
We point out that the theorems are not proved in the order they are
stated but instead grouped based on the approach and techniques
they rely on, which makes the reading of the proofs somewhat easier.


\section{Proof of Theorem~\texorpdfstring{\protect\ref{th:1D}}{4}}
\label{sec:1D}

We first study the one-dimensional nearest neighbor system.
The analysis in this case relies on the study of the process that
keeps track of the boundaries between the two strategies and strongly
differs from the analysis of the system in higher spatial dimensions
or with a larger range of interactions.
Throughout this section, we let $\xi_t$ denote the \emph{boundary process}
\[
\xi_t (x):= \eta_t (x + 1/2) - \eta_t (x
- 1/2) \qquad\mbox{for all } x \in\D:= \Z+ 1/2
\]
and think of sites in state~0 as empty, sites in state~$+1$ as occupied
by a~$+$~particle and sites in state~$-1$ as occupied by
a~$-$~particle.
To begin with, we assume that
%
\begin{equation}
\label{eq:config-left} \eta_0 \dvtx \Z\longrightarrow\{1, 2 \} \qquad\mbox{is
such that } \eta_0 (x) = 1 \mbox{ for all } x \leq M,
\end{equation}
where $M$ is a positive integer.
Note that, from the point of view of the boundary process, this
implies that we start with no particle to the left of $M$.
Also, we let
\[
X_t:= \inf\bigl\{x \in\D\dvtx \xi_t (x) \neq0 \bigr\}
= \inf\bigl\{x \in \D\dvtx \xi_t (x) = 1 \bigr\}
\]
denote the position of the leftmost particle, which is necessarily a
$+$ particle in view of the initial configuration.
The key to proving the first part of Theorem~\ref{th:1D} is given by
the next lemma, which shows the result when starting from the
particular initial configuration \eqref{eq:config-left}.

%
\begin{lemma}
\label{lem:1D-drift}
Assume \eqref{eq:config-left} and $a_{11} > \max(a_{22}, a_{21}) +
(a_{21} - a_{12})$. Then
\begin{eqnarray}
P \Bigl(X_t > 0 \mbox{ for all } t > 0
\mbox{ and } \lim_{t
\to\infty} X_t = \infty\Bigr) \geq1 -
\exp(- a_0 M)
\nonumber\\
\eqntext{\mbox{for some } a_0 > 0.}
\end{eqnarray}
\end{lemma}

\begin{pf}
The idea is to prove that there exists $a > 0$ such that
\[
\lim_{h \to0} h^{-1} E
(Z_{t + h} - Z_t | Z_t) \leq0\qquad \mbox{almost
surely } \mbox{where } Z_t:= \exp (- a X_t)
\]
and then apply the optimal stopping theorem to the supermartingale $(Z_t)$.
Since the transition rates of the process $(X_t)$ depends on the
distance between the~$+$~particle it keeps track of and the next
particle to the right, we introduce the gap process
\begin{eqnarray*}
K_t&:=& \inf\bigl\{x \in\D\dvtx x >
X_t \mbox{ and } \xi_t (x) \neq0 \bigr\} -
X_t
\\
&=& \inf\bigl\{x \in\D\dvtx x > X_t \mbox{ and } \xi_t (x)
= -1 \bigr\} - X_t.
\end{eqnarray*}
Looking at the payoff of the players at sites $X_t \pm1/2$ and their
associated birth and death rates, which are reported
in Figure~\ref{fig:interface} depending on the value of the gap, we find
%
\begin{eqnarray}
\label{eq:1D-drift-4} %
\Phi_1 (a) &:= &
\lim_{h \to0} h^{-1} E (Z_{t + h} -
Z_t | Z_t, K_t = 1)
\nonumber\\
& \leq& \bigl(e^{- a (X_t + 2)} - e^{- a X_t}\bigr) \lim
_{h \to0} h^{-1} P (X_{t + h} - X_t
\geq2 | K_t = 1)
\nonumber\\
& &{} + \bigl(e^{- a (X_t - 1)} - e^{- a X_t}\bigr) \lim_{h \to
0}
h^{-1} P (X_{t + h} - X_t = - 1 | K_t
= 1)
\nonumber
\\[-8pt]
\\[-8pt]
\nonumber
& = & Z_t \bigl(e^{- 2a} - 1\bigr) \bigl((1/4) \max(0,
a_{11} + a_{12}) + \max (0, - a_{21})\bigr)
\nonumber\\
& &{} + Z_t \bigl(e^a - 1\bigr) \bigl((1/4) \max(0, -
a_{11} - a_{12}) + (1/2) \max(0, a_{21})\bigr) \nonumber\\
&=:&
\Psi_1 (a)\nonumber
\end{eqnarray}

%
\begin{figure}

\includegraphics{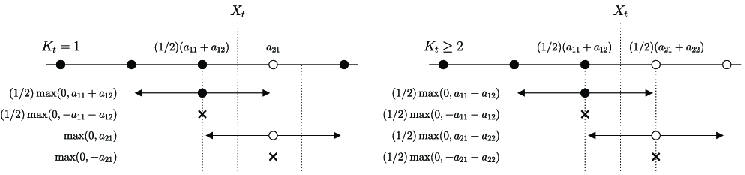}

\caption{Picture related to the proof of Lemma~\protect\ref{lem:1D-drift}.
The numbers are the top of both pictures are the payoffs of the two
players at $X_t \pm1/2$ and the numbers at the bottom the associated
birth and death rates.
Birth events are symbolically represented by double arrows and death
events by crosses.}
\label{fig:interface}
\end{figure}

\noindent almost surely, where the inequality is obtained by ignoring births from
$X_t + 3/2$ and jumps of more than two units to the right.
Note also that the derivative of the right-hand side evaluated at~$a =
0$ satisfies the following inequality almost surely:
%
\begin{eqnarray}
\label{eq:1D-drift-5} %
\Psi'_1
(0) & = & - (1/2) \max(0, a_{11} + a_{12}) Z_t -
2 \max(0, - a_{21}) Z_t\nonumber
\\
&&{} + (1/4) \max(0, - a_{11} - a_{12}) Z_t +
(1/2) \max(0, a_{21}) Z_t\nonumber
\\
& \leq& (1/4) \max(0, - a_{11} - a_{12}) Z_t -
(1/4) \max(0, a_{11} + a_{12}) Z_t
\\
&&{} + (1/2) \max(0, a_{21}) Z_t - (1/2) \max(0, -
a_{21}) Z_t\nonumber
\\
& = & (1/4) (- a_{11} - a_{12} + 2 a_{21})
Z_t\nonumber
\\
& \leq& (1/4) \bigl(\max(a_{22}, a_{21}) + (a_{21}
- a_{12}) - a_{11}\bigr) Z_t < 0.\nonumber
\end{eqnarray}
Similarly, conditioning on the event $K_t \geq2$, we have
%
\begin{eqnarray}
\label{eq:1D-drift-6} %
\Phi_2 (a) &:= &
\lim_{h \to0} h^{-1} E (Z_{t + h} -
Z_t | Z_t, K_t \geq2)\nonumber
\\
& = & \bigl(e^{- a (X_t + 1)} - e^{- a X_t}\bigr) \lim_{h \to0}
h^{-1} P (X_{t + h} - X_t = 1 | K_t
\geq2)\nonumber
\\
& &{} + \bigl(e^{- a (X_t - 1)} - e^{- a X_t}\bigr) \lim_{h \to
0}
h^{-1} P (X_{t + h} - X_t = - 1 | K_t
\geq2)
\\
& = & Z_t \bigl(e^{-a} - 1\bigr) (1/4) \bigl(\max(0,
a_{11} + a_{12}) + \max(0, - a_{22} -
a_{21})\bigr)\nonumber
\\
& & + Z_t \bigl(e^a - 1\bigr) (1/4) \bigl(\max(0, -
a_{11} - a_{12}) + \max(0, a_{22} +
a_{21})\bigr) \nonumber
\end{eqnarray}
almost surely. Taking again the derivative at $a = 0$, we get
%
\begin{eqnarray}
\label{eq:1D-drift-7} %
\Phi'_2
(0) & = & - (1/4) \max(0, a_{11} + a_{12}) Z_t -
(1/4) \max(0, - a_{22} - a_{21}) Z_t\nonumber
\\
& &{} + (1/4) \max(0, - a_{11} - a_{12}) Z_t +
(1/4) \max(0, a_{22} + a_{21}) Z_t
\nonumber
\\[-8pt]
\\[-8pt]
\nonumber
& = & (1/4) (- a_{11} - a_{12} + a_{22} +
a_{21}) Z_t
\\
& \leq& (1/4) \bigl(\max(a_{22}, a_{21}) + (a_{21}
- a_{12}) - a_{11}\bigr) Z_t < 0 \nonumber
\end{eqnarray}
almost surely.
From \eqref{eq:1D-drift-4}--\eqref{eq:1D-drift-7}, we deduce that
\[
\Phi_1 (a_0) \leq\Psi_1
(a_0) \leq\Psi_1 (0) = 0\quad \mbox{and}\quad \Phi_2
(a_0) \leq\Phi_2 (0) = 0
\]
for some $a_0 > 0$ fixed from now on.
In particular,
\[
\lim_{h \to0} h^{-1} E
\bigl(\exp(- a_0 X_{t +
h}) - \exp(- a_0
X_t) | X_t\bigr) \leq0 \qquad\mbox{almost surely},
\]
which shows that $Z_t = \exp(- a_0 X_t)$ is a supermartingale.
As mentioned above, we conclude using the optimal stopping theorem: we
introduce the stopping times
\[
\tau_-:= \inf\{t \dvtx X_t = -1/2 \} \quad\mbox{and}\quad
 \tau_n:= \inf\{t \dvtx X_t \geq n + 1/2 \}\qquad \mbox{for all } n > M.
\]
Using that $T_n:= \min(\tau_-, \tau_n)$ is almost surely finite, we get
\begin{eqnarray*} E Z_{T_n} & \leq& E Z_0 =
E e^{- a_0 X_0} \leq e^{- a_0
M},
\\
E Z_{T_n} & = & E (Z_{T_n} | T_n = \tau_-) P
(T_n = \tau _-) + E (Z_{T_n} | T_n =
\tau_n) P (T_n = \tau_n)
\\
& \geq& e^{a_0 / 2} \bigl(1 - P (T_n = \tau_n)
\bigr) + e^{- a_0 n} P (T_n = \tau_n).
\end{eqnarray*}
Observing that the sequence $\{T_n = \tau_n \}$ is nonincreasing for
the inclusion, applying the monotone convergence theorem and using the previous
inequalities, we deduce that
\begin{eqnarray*} &&P \Bigl(X_t > 0 \mbox{ for all } t > 0
\mbox{ and } \lim_{t \to\infty
} X_t = \infty\Bigr)
\\
&&\qquad\geq P (T_n = \tau _n \mbox{ for all } n > M) = \lim
_{n \to\infty} P (T_n = \tau_n) \\
&&\qquad\geq\lim
_{n \to
\infty} \bigl(e^{a_0 / 2} - e^{- a_0 M}\bigr)
\bigl(e^{a_0 / 2} - e^{-
a_0 n}\bigr)^{-1} \geq1 -
e^{- a_0 M}.
\end{eqnarray*}
This completes the proof of the lemma.
\end{pf}

It follows from the previous lemma that, starting more generally from
a product measure with a positive density of type~1 players, strategy~1
wins with probability one.
This statement, which corresponds to the first part of Theorem~\ref
{th:1D}, is proved in the next lemma.

%
\begin{lemma}
\label{lem:1D-wins}
Assume \eqref{eq:product} and $a_{11} > \max(a_{22}, a_{21}) +
(a_{21} - a_{12})$. Then strategy~1 wins.
\end{lemma}

\begin{pf}
Let $M$ be a positive integer, and let
\[
x_M:= \inf\bigl\{x \in\Z\dvtx x > 0 \mbox{ and } \eta_0
(x) = \eta_0 (x + 1) = \cdots= \eta_0 (x + 2M + 2) = 1
\bigr\}.
\]
Note that, starting from \eqref{eq:product}, vertex $x_M$ is well
defined and almost surely finite.
We define the cluster starting at $x_M$ as the set of space--time
points that can be reached from $x_M$ by a path moving forward in time and
contained in the space--time region occupied by type~1 players
\[
C (x_M):= \bigl\{(x, t) \in\Z\times\R_+ \dvtx (x_M, 0)
\to(x, t) \bigr\},
\]
where $(x_M, 0) \to(x, t)$ means that there exist
\[
z_0 = x_M, z_1, \ldots, z_n:=
x \in\Z\quad\mbox{and} \quad s_0:= 0 < s_1 < \cdots<
s_n < s_{n + 1}:= t \in\R_+
\]
such that the following two conditions hold:
\begin{itemize}
\item for $i = 1, 2, \ldots, n$, we have $|z_i - z_{i - 1}| = 1$, and

\item for $i = 0, 1, \ldots, n$, we have $\eta_s (x_i) = 1$ for all
$s_i \leq s \leq s_{i + 1}$.
\end{itemize}
Finally, for all times $t$, we let
\[
l_t:= \inf\bigl\{x \in\Z\dvtx (x,
t) \in C (x_M) \bigr\} \quad\mbox{and}\quad r_t:= \sup\bigl\{x
\in\Z\dvtx (x, t) \in C (x_M) \bigr\}
\]
be, respectively, the leftmost and the rightmost vertices in the cluster.
Due to one-dimensional nearest neighbor interactions, as long as the
cluster is nonempty, all vertices between the leftmost and rightmost
vertices follow strategy~1.
In particular, it follows from Lemma~\ref{lem:1D-drift} and the
obvious symmetry of the evolution rules that the probability that
strategy~1 wins
is larger than
\begin{eqnarray*}
&&P \bigl((x, t) \in C (x_M) \mbox{ for all $x \in\Z$ and all $t$ large}\bigr)
\\
&&\qquad= P \Bigl(l_t < r_t \mbox{ for all } t > 0 \mbox{ and }
\lim_{t \to \infty} l_t = - \infty\mbox{ and } \lim
_{t \to\infty} r_t = + \infty\Bigr) \\
&&\qquad\geq P \Bigl(l_t < x_M + M + 1 < r_t \mbox{ for all }
t > 0 \mbox{ and}\\
&&\hspace*{68pt}{} \lim_{t \to\infty} l_t = - \infty
\mbox{ and } \lim_{t \to\infty} r_t = + \infty\Bigr)
\\
&&\qquad\geq P \Bigl((- X_t) > 0 \mbox{ for all } t > 0 \mbox{ and } \lim
_{t \to \infty} (- X_t) = - \infty\Bigr)
\\
&&\qquad\quad{}\times P \Bigl(X_t > 0 \mbox{ for all } t > 0 \mbox{ and } \lim
_{t \to\infty
} X_t = + \infty\Bigr)\\
&&\qquad \geq\bigl(1 - \exp(-
a_0 M)\bigr)^2.
\end{eqnarray*}
Since $a_0 > 0$, and this holds for all $M$, it follows that strategy~1
wins almost surely.
This completes the proof of the lemma and the first part of the theorem.
\end{pf}

The next two lemmas focus on the second part of the theorem whose
proof consists in showing extinction of the boundary process
starting from any translation invariant distribution.

%
\begin{lemma}
\label{lem:1D-cluster-1}
Let $A \in\mathcal M_2^*$ as in \eqref{eq:measure-zero}.
Then the system clusters if
%
\begin{equation}
\label{eq:1D-cluster-1} a_{22} + a_{21} < a_{11} +
a_{12} \quad\mbox{and}\quad (a_{11} + a_{12} > 0 \mbox{ or }
a_{21} < 0).
\end{equation}
\end{lemma}

\begin{pf}
As pointed out before the statement of the lemma, to prove clustering,
it suffices to prove extinction of the boundary process
since, for all $x < y$, we have
\begin{eqnarray*} P \bigl(\eta_t (x) \neq
\eta_t (y)\bigr) & \leq& P \bigl(\eta_t (z) \neq
\eta_t (z + 1) \mbox{ for some } z = x, \ldots, y - 1\bigr)
\\
& \leq& \sum_{z = x}^{y - 1} P \bigl(
\eta_t (z) \neq\eta_t (z + 1)\bigr) = \sum
_{z = x}^{y - 1} P \bigl(\xi_t (z + 1/2) \neq0
\bigr),
\end{eqnarray*}
which converges to zero whenever the boundary process goes extinct.
In particular, the main objective is to show that the density of
particles $u (t)$ in the boundary process at time $t$ converges to zero
as time goes to infinity, that is,
%
\begin{equation}
\label{eq:1D-cluster-1a} u (t):= P \bigl(\xi_t (x) \neq0\bigr) \to0
\qquad\mbox{as } t \to\infty.
\end{equation}
Due to translation invariance of the initial distribution and the
evolution rules, the probability above is indeed constant across space.
The definition of the boundary process also implies that two
consecutive particles must have opposite signs.
Moreover, due to one-dimensional nearest neighbor interactions,
particles cannot be created and if a particle jumps onto another
particle then both
particles, necessarily with opposite signs, annihilate.
In particular,
\[
 u (t) \leq u (s) \qquad\mbox{for all } s \leq t
\mbox{ therefore } \lim_{t \to\infty} u (t):= l \mbox { exists}.
\]
To show that the limit $l = 0$, we prove that, in every group of four
consecutive particles at arbitrary times, at least one particle is
killed after
an almost surely finite time.
Let $s \geq0$ and
\begin{eqnarray*}X_s^+ &:= & \inf\bigl\{x \in\D
\dvtx x > 0 \mbox{ and } \xi_s (x) = +1 \bigr\},
\\
X_s^- &:= & \inf\bigl\{x \in\D\dvtx x > X_s^+
\mbox{ and } \xi_s (x) = -1 \bigr\}
\end{eqnarray*}
be the position at time~$s$ of the first~$+$~particle to the right of
the origin and the position at time~$s$ of the following particle,
which is necessarily a~$-$~particle.
Also, we let
\begin{eqnarray*}X_t^+ &:= & \mbox{position at time
$t > s$ of the $+$ particle that originates}
\\
&&\mbox{from $X_s^+$ at time
$s$},
\\
X_t^- &:= & \mbox{position at time $t > s$ of the $-$ particle
that originates}\\
&&\mbox{from $X_s^-$ at time $s$}
\end{eqnarray*}
which are well defined until one particle is killed when we set
$X_t^{\pm} = \varnothing$, and
\[
\tau_+:= \inf\bigl\{t > s \dvtx X_t^+ = \varnothing\bigr\}
\quad\mbox{and}\quad \tau_-:= \inf\bigl\{t > s \dvtx X_t^- = \varnothing\bigr
\}.
\]
We claim that $\inf(\tau_+, \tau_-) < \infty$.
To prove our claim, we let
\begin{eqnarray*}\sigma_+ &:= & \mbox{time at which the $+$
particle at $X_t^+$ annihilates with a $-$ particle}
\\
&&\mbox{on its left},
\\
\sigma_- &:= & \mbox{time at which the $-$ particle at $X_t^-$
annihilates}\\
&&\mbox{ with a $+$ particle on its right}.
\end{eqnarray*}
By inclusion of events, we have
%
\begin{equation}
\label{eq:1D-cluster-1b} P \bigl(\inf(\tau_+, \tau_-) < \infty| \inf(\sigma_+, \sigma _-)
< \infty\bigr) = 1.
\end{equation}
Moreover, in view of the first inequality in \eqref{eq:1D-cluster-1},
we have the following transition rate:
\begin{eqnarray*}
&&\lim_{h \to0} h^{-1} P
\bigl(X_{t + h}^+ - X_t^+ = 1 | X_t^+ \neq
\varnothing\mbox{ and } \xi_t \bigl(X_t^+ - 1\bigr) =
\xi_t \bigl(X_t^+ + 1\bigr) = 0\bigr)
\\
&&\qquad= (1/4) \bigl(\max(0, a_{11} + a_{12}) + \max(0, -
a_{22} - a_{21})\bigr)
\\
&&\qquad> (1/4) \bigl(\max(0, a_{11} + a_{12}) - (a_{11} +
a_{12})\\
&&\hspace*{59pt}{} + \max(0, - a_{22} - a_{21}) +
(a_{22} + a_{21})\bigr)
\\
&&\qquad= (1/4) \bigl(\max(0, - a_{11} - a_{12}) + \max(0,
a_{22} + a_{21})\bigr)
\\
&&\qquad= \lim_{h \to0} h^{-1} P \bigl(X_{t +
h}^+ -
X_t^+ = - 1 | X_t^+ \neq\varnothing\mbox{ and}\\
&&\hspace*{89pt}{}\xi_t \bigl(X_t^+ - 1\bigr) = \xi_t
\bigl(X_t^+ + 1\bigr) = 0\bigr).
\end{eqnarray*}
Since on the event $\sigma_+ = \infty$ the particle at $X_t^+$ cannot
jump onto a $-$~particle on its left, we deduce that,
on this event, the position of the particle has a positive drift until
it is one unit from the~$-$~particle on its right.
Similarly, on the event $\sigma_- = \infty$, the position of the
particle at $X_t^-$ has a negative drift until it is one unit from the
$+$~particle on its left.
This implies that
%
\begin{eqnarray}
\label{eq:1D-cluster-1c} %
&&P \bigl(X_t^- -
X_t^+ = 1 \mbox{ and }\xi _t \bigl(X_t^- + 1\bigr) \xi_t
\bigl(X_t^+ - 1\bigr) = 0
\nonumber
\\[-8pt]
\\[-8pt]
\nonumber
&&\hspace*{49pt} \mbox{for some } t \in(s, \infty) | \inf(
\sigma_+, \sigma_-) = \infty\bigr) = 1.
\end{eqnarray}
Also, each time $X_t^- - X_t^+ = 1$, both particles annihilate at rate
at least
\[
(1/2) \max(0, - a_{21}) + (1/4) \max(0, a_{11} +
a_{12}).
\]
The second inequality in \eqref{eq:1D-cluster-1} implies that this rate
is strictly positive which, together with \eqref{eq:1D-cluster-1c}
and the fact that the process is Markov further implies that the two
particles at $X_t^{\pm}$ annihilate after an almost surely finite time.
In particular,
%
\begin{equation}
\label{eq:1D-cluster-1d} P \bigl(\inf(\tau_+, \tau_-) < \infty| \inf(\sigma_+, \sigma _-)
= \infty\bigr) = 1.
\end{equation}
Combining \eqref{eq:1D-cluster-1b} and \eqref{eq:1D-cluster-1d}, we
deduce that, in every group of four consecutive particles at arbitrary times,
at least one particle is killed after a finite time, therefore there
exists a strictly increasing sequence of almost surely finite times
$s_0 = 0 < s_1 < \cdots< s_n < \cdots$ such that
\[
u (s_n) \leq(1/2) u (s_{n - 1}) \leq(1/4) u
(s_{n - 2}) \leq\cdots \leq(1/2)^n u (s_0)
\leq(1/2)^n.
\]
This shows \eqref{eq:1D-cluster-1a} and completes the proof of the lemma.
\end{pf}

To complete the proof of the theorem, the last step is to prove the
analog of Lemma~\ref{lem:1D-cluster-1} when the second set of
inequalities in \eqref{eq:1D-cluster-1}
does not hold.
This includes in particular all the payoff matrices that satisfy \eqref
{eq:clustering}.
This case is rather delicate since a player of either type cannot
change her strategy whenever her two nearest neighbors and next two
nearest neighbors all four follow the
same strategy.
In particular, two particles next to each other annihilate at a
positive rate only if there is a third particle nearby so the idea of
the proof is to show that we can indeed
bring sets of three consecutive particles together.
Figure~\ref{fig:cluster} gives an illustration of this problem:
boundaries by pair repulse each other and at least three particles are
necessary to induce annihilation.
%
\begin{figure}[b]
\centering
\begin{tabular}{@{}cc@{}}

\includegraphics{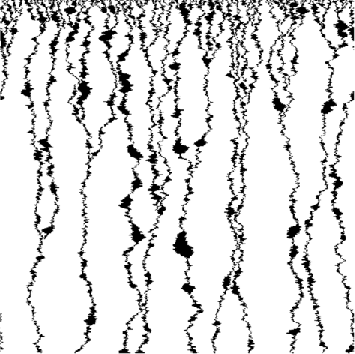}  & \includegraphics{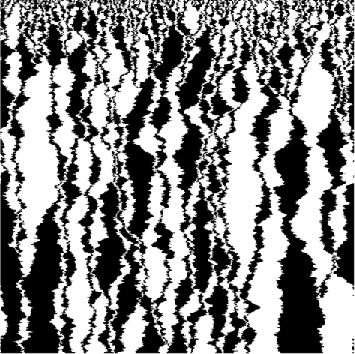}\\
\footnotesize{(a) $A = ((-8, 3), (4, -8))$} & \footnotesize{(b) $A =
((-8, 4), (4, -8))$}
\end{tabular}
\caption{Realizations of the one-dimensional nearest neighbor
spatial game on the torus $\Z/ 600 \Z$ for two different payoff matrices
that satisfy the inequalities in \protect\eqref{eq:clustering}.
Time goes down until time 10,000.}
\label{fig:cluster}
\end{figure}

%
\begin{lemma}
\label{lem:1D-cluster-2}
Let $A \in\mathcal M_2^*$ as in \eqref{eq:measure-zero}.
Then the system clusters if
%
\begin{equation}
\label{eq:1D-cluster-2} a_{22} + a_{21} < a_{11} +
a_{12} < 0 \quad\mbox{and}\quad a_{21} > 0.
\end{equation}
\end{lemma}

\begin{pf}
Following the same approach as in the previous lemma, it suffices to
prove that, starting with a positive density of boundaries,
annihilating events occur in a finite time within a given finite set
of consecutive boundaries.
The main difficulty is that condition \eqref{eq:1D-cluster-2} now
implies that starting with a single type~2 player, the two
resulting boundaries cannot annihilate therefore to prove the
occurrence of annihilating events, we need to look at a set of four
boundaries instead of two like in the proof of the previous lemma.
To begin with, we start from a configuration with infinitely many
type~1 players and exactly four boundaries, which forces the
initial number of type~2 players to be finite, and denote the position
of the boundaries by
\[
X_t^+ < X_t^- < Y_t^+ <
Y_t^-
\]
before an annihilating event has occurred.
The same argument as in the proof of the previous lemma based on the
first inequality in \eqref{eq:1D-cluster-2} implies that
%
\begin{eqnarray}
\label{eq:1D-cluster-2a} %
&&\lim_{h \to0}
h^{-1} E \bigl(\bigl(X_{t + h}^- - X_{t +
h}^+\bigr) -
\bigl(X_t^- - X_t^+\bigr) |
\nonumber
\\[-8pt]
\\[-8pt]
\nonumber
&&\hspace*{29pt}\qquad X_t^- - X_t^+ > 1 \mbox{ and } Y_t^+ -
X_t^- > 1\bigr) < 0.
\end{eqnarray}
The same applies to the two rightmost boundaries,
%
\begin{eqnarray}
\label{eq:1D-cluster-2b} %
&& \lim_{h \to0}
h^{-1} E \bigl(\bigl(Y_{t + h}^- - Y_{t +
h}^+\bigr) -
\bigl(Y_t^- - Y_t^+\bigr) |
\nonumber
\\[-8pt]
\\[-8pt]
\nonumber
&&\qquad \hspace*{29pt} Y_t^- - Y_t^+ > 1 \mbox{ and } Y_t^+ -
X_t^- > 1\bigr) < 0.
\end{eqnarray}
Moreover, by symmetry, we have
%
\begin{eqnarray}
\label{eq:1D-cluster-2c} %
&& \lim_{h \to0}
h^{-1} P \bigl(\bigl(X_{t + h}^- + X_{t + h}^+\bigr) -
\bigl(X_t^- + X_t^+\bigr) = 1 | Y_t^+ -
X_t^- > 1\bigr)
\nonumber
\\[-8pt]
\\[-8pt]
\nonumber
&&\qquad= \lim_{h \to0} h^{-1} P \bigl(\bigl(X_{t + h}^-
+ X_{t + h}^+\bigr) - \bigl(X_t^- + X_t^+\bigr)
= - 1 | Y_t^+ - X_t^- > 1\bigr).
\end{eqnarray}
In words, the midpoint between the two leftmost boundaries evolve
according to a symmetric random walk.
The same holds for the midpoint between the two rightmost boundaries.
To deduce the occurrence of an annihilating event, we distinguish two cases:
\begin{itemize}
\item\emph{Case} 1. Assume \eqref{eq:1D-cluster-2} and $a_{12} < 0$.
In this case, \eqref{eq:1D-cluster-2c} and the recurrence of
one-dimensional symmetric simple random walks imply that
\[
P \bigl(Y_t^+ - X_t^- = 1 \mbox{ for some } t > 0\bigr)
= 1.
\]
Since the event above induces a configuration in which a type~1 player
has two type~2 neighbors, and so a negative payoff
$a_{12} < 0$, each time this event occurs, the two intermediate
boundaries annihilate at a positive rate.
This, together with a basic restart argument, implies the occurrence
of an annihilating event after an almost surely finite time.
\item\emph{Case} 2. Assume \eqref{eq:1D-cluster-2} and $a_{12} > 0$.
In this case, \eqref{eq:1D-cluster-2a}--\eqref{eq:1D-cluster-2c} imply
that, with probability one, we can bring three consecutive
boundaries together: more precisely,
\[
P \bigl(Y_t^+ - X_t^- = 1 \mbox{ and }
\bigl(X_t^- - X_t^+ = 1 \mbox{ or } Y_t^- -
Y_t^+ = 1\bigr) \mbox{ for some } t > 0\bigr) = 1.
\]
Since the event above induces a configuration in which a type~2 player
has two type~1 neighbors, and so a positive payoff
$a_{21} > 0$, each time this event occurs, either the two leftmost
boundaries or the two rightmost boundaries annihilate at a
positive rate.
We again deduce the occurrence of an annihilating event after an
almost surely finite time.
\end{itemize}
The two results above still hold when starting from a translation
invariant distribution with a positive density of boundaries unless the leftmost
of the four boundaries or the rightmost of the four boundaries
annihilate before with another boundary.
In any case, each set of four consecutive boundaries is reduced by one
after an almost surely finite time.
This, together with the exact same arguments as in the proof of the
previous lemma, establishes the desired result.
\end{pf}

Lemmas \ref{lem:1D-cluster-1}--\ref{lem:1D-cluster-2} imply clustering
for all $A \in\mathcal M_2^*$ with $a_{11} + a_{12} > a_{22} + a_{21}$.
The second part of the theorem directly follows by also using some
obvious symmetry.


\section{Proof of Theorem~\texorpdfstring{\protect\ref{th:biased}}{1}}
\label{sec:biased}

This section is devoted to the proof of Theorem~\ref
{th:biased}, which relies on a standard coupling argument between
the spatial game and a biased voter model that favors individuals of type~1.
Recall that the biased voter model \cite{bramson_griffeath_1980,bramson_griffeath_1981} is the spin system with flip rate
\[
c_{\mathrm{BV}} (x, \xi) = \mu_1 f_1 (x, \xi) \ind
\bigl\{\xi(x) = 2 \bigr\} + \mu _2 f_2 (x, \xi) \ind\bigl
\{\xi(x) = 1 \bigr\}
\]
for which individuals of type~1 win whenever $\mu_1 > \mu_2$.
Recall also that the spatial game reduces to such a spin system if and
only if the payoff received by players of either type is
constant regardless of the spatial configuration.
In particular, strategy~1 wins whenever
\[
a_{11} = a_{12} > a_{21} = a_{22}.
\]
For all other parameters, the dynamics is more complicated but the
process can be coupled with a biased voter model that favors
type~1 individuals in a certain parameter region.
To make this argument rigorous and prove Theorem~\ref{th:biased}, we
introduce the payoff functions
\begin{eqnarray*}
\phi_1 (z) &:= & a_{12}
(z / N) + a_{11} (1 - z / N) = (a_{12} - a_{11}) (z
/ N) + a_{11},
\\
\phi_2 (z) &:= & a_{22} (z / N) + a_{21} (1 - z
/ N) = (a_{22} - a_{21}) (z / N) + a_{21}.
\end{eqnarray*}
The coupling argument is given in the proof of the following lemma.

%
\begin{lemma}
\label{lem:biased-voter}
Assume that $a_{12} > a_{21}$.
Then, strategy~1 wins whenever
%
\begin{equation}
\label{eq:biased-voter-0} \max\bigl(\phi_2 (z) \dvtx z \in\{0, 1, \ldots, N - 1
\}\bigr) < \min\bigl(\phi _1 (z) \dvtx z \in\{1, 2, \ldots, N \}
\bigr).
\end{equation}
\end{lemma}

\begin{pf}
Denoting by $c_{\mathrm{SG}} (x, \eta)$ the flip rate of the spatial game, we have
%
\begin{equation}
\label{eq:biased-voter-1} c_{\mathrm{SG}} (x, \eta) = c_{\mathrm{BV}} (x, \xi) = 0
\qquad\mbox{when } f_{\eta(x)} (x, \eta) = f_{\xi(x)} (x, \xi) = 1.
\end{equation}
Now, observe that the player at $x$ may flip $2 \to1$ because she has
a negative payoff and so a positive death rate or because
she has a neighbor following strategy~1 that has a positive payoff and
so a positive birth rate.
In particular, given that the player at vertex $x$ follows strategy~2
and has at least one neighbor following strategy~1, the rate
at which the strategy at $x$ flips is
%
\begin{eqnarray}
\label{eq:biased-voter-2} %
&&\max\bigl(0, - \phi(x, \eta)\bigr)
f_1 (x, \eta) + N^{-1} \sum_{y \sim
x}
\max\bigl(0, \phi(y, \eta)\bigr) \ind\bigl\{\eta(y) = 1 \bigr\}
\nonumber
\\[-8pt]
\\[-8pt]
\nonumber
&&\qquad\geq\min_{z \neq N} \max\bigl(0, - \phi_2 (z)\bigr)
f_1 (x, \eta) + \min_{z
\neq0} \max\bigl(0,
\phi_1 (z)\bigr) f_1 (x, \eta).
\end{eqnarray}
Similarly, given that the player at vertex $x$ follows strategy~1 and
has at least one neighbor following strategy~2,
the rate at which the strategy at $x$ flips is
%
\begin{eqnarray}
\label{eq:biased-voter-3} %
&& \max\bigl(0, - \phi(x, \eta)\bigr)
f_2 (x, \eta) + N^{-1} \sum_{y \sim
x}
\max\bigl(0, \phi(y, \eta)\bigr) \ind\bigl\{\eta(y) = 2 \bigr\}
\nonumber
\\[-8pt]
\\[-8pt]
\nonumber
&&\qquad\leq\max_{z \neq0} \max\bigl(0, - \phi_1 (z)\bigr)
f_2 (x, \eta) + \max_{z
\neq N} \max\bigl(0,
\phi_2 (z)\bigr) f_2 (x, \eta).
\end{eqnarray}
Combining \eqref{eq:biased-voter-1}--\eqref{eq:biased-voter-3}, we
obtain that strategy~1 wins whenever
%
\begin{eqnarray}
\label{eq:biased-voter-4} %
\mu_2&:=& \max
_{z \neq0} \max\bigl(0, - \phi_1 (z)\bigr) + \max
_{z \neq
N} \max\bigl(0, \phi_2 (z)\bigr)
\nonumber
\\[-8pt]
\\[-8pt]
\nonumber
&<& \min_{z
\neq N} \max\bigl(0, - \phi_2 (z)\bigr) +
\min_{z \neq0} \max \bigl(0, \phi_1 (z)\bigr) =\dvtx
\mu_1
\end{eqnarray}
since, under this assumption, if $\eta(x) \leq\xi(x)$ for all $x \in
\Z^d$, then
\begin{eqnarray*} c_{\mathrm{SG}} (x, \eta)& \leq&\mu_2
f_2 (x, \xi) = c_{\mathrm{BV}} (x, \xi) \qquad \mbox{when }  \eta(x) =
\xi(x) = 1,
\\
c_{\mathrm{SG}} (x, \eta) &\geq&\mu_1 f_1 (x, \xi) =
c_{\mathrm{BV}} (x, \xi) \qquad \mbox{when } \eta(x) = \xi(x) = 2,
\end{eqnarray*}
which, according to Theorem III.1.5 in \cite{liggett_1985}, implies
that the set of type~1 players dominates stochastically its
counterpart in a biased voter model that favors type~1.
To complete the proof, it remains to show that \eqref
{eq:biased-voter-0} implies \eqref{eq:biased-voter-4}.
Note that \eqref{eq:biased-voter-4} is equivalent to
%
\begin{eqnarray}
\label{eq:biased-voter-5} %
&&\max_{z \neq N} \max
\bigl(0, \phi_2 (z)\bigr) - \min_{z \neq N} \max \bigl(0,
- \phi_2 (z)\bigr)
\nonumber
\\[-8pt]
\\[-8pt]
\nonumber
&&\qquad< \min_{z \neq0} \max\bigl(0, \phi_1 (z)\bigr) - \max
_{z \neq0} \max \bigl(0, - \phi_1 (z)\bigr).
\end{eqnarray}
Note also that the left-hand side of \eqref{eq:biased-voter-5} reduces to
%
\begin{eqnarray}
\label{eq:biased-voter-6} %
&&\max_{z \neq N} \max
\bigl(0, \phi_2 (z)\bigr) - \min_{z \neq N} \max \bigl(0,
- \phi_2 (z)\bigr)\nonumber \\
&&\qquad= \max\Bigl(0, \max_{z
\neq N}
\phi_2 (z)\Bigr) - \max\Bigl(0, \min_{z \neq
N} \bigl(-
\phi_2 (z)\bigr)\Bigr)
\nonumber
\\[-8pt]
\\[-8pt]
\nonumber
&&\qquad = \max\Bigl(0, \max_{z \neq N}
\phi_2 (z)\Bigr) - \max\Bigl(0, - \max_{z
\neq N}
\phi_2 (z)\Bigr)
\\
&&\qquad= \max\Bigl(0, \max_{z \neq N} \phi_2 (z)\Bigr) + \min
\Bigl(0, \max_{z \neq
N} \phi_2 (z)\Bigr) = \max
_{z \neq N} \phi_2 (z). \nonumber
\end{eqnarray}
Similarly, for the payoff of type~1 players, we have
%
\begin{equation}
\label{eq:biased-voter-7} %
 \min_{z \neq0} \max
\bigl(0, \phi_1 (z)\bigr) - \max_{z \neq0} \max
\bigl(0, - \phi_1 (z)\bigr) = \min_{z \neq0}
\phi_1 (z).
\end{equation}
Since \eqref{eq:biased-voter-0}, \eqref{eq:biased-voter-6} and \eqref
{eq:biased-voter-7} imply \eqref{eq:biased-voter-5} and
then \eqref{eq:biased-voter-4}, the proof is complete.
\end{pf}

In the following lemma, we complete the proof of Theorem~\ref
{th:biased} based on \eqref{eq:biased-voter-0}.

%
\begin{lemma}
\label{lem:biased-case}
Assume that $a_{12} > a_{21}$.
Then, strategy~1 wins whenever
%
\begin{equation}
\label{eq:biased-case-0} \max(a_{22}, a_{21}) + a_{21} (N -
1)^{-1} < \min (a_{11}, a_{12}) + a_{12}
(N - 1)^{-1}.
\end{equation}
\end{lemma}

\begin{pf}
This directly follows from Lemma~\ref{lem:biased-voter} by showing
that the parameter region in which the inequality
in \eqref{eq:biased-voter-0} holds is exactly \eqref{eq:biased-case-0}.
To re-write \eqref{eq:biased-voter-0} explicitly in terms of the
payoffs, we distinguish four cases depending on the
monotonicity of the functions $\phi_1$ and $\phi_2$.
\begin{itemize}
\item\emph{Case} 1. When $\max(a_{22}, a_{21}) = a_{21}$ and $\min
(a_{11}, a_{12}) = a_{12}$, both payoff
functions are decreasing; therefore, according to the previous lemma,
strategy~1 wins whenever
\[
 \max_{z \neq N} \phi_2
(z) = \phi_2 (0) = a_{21} < a_{12} =
\phi_1 (N) = \min_{z \neq0} \phi_1 (z),
\]
which is always true under our general assumption $a_{12} > a_{21}$.

\item\emph{Case} 2. When $\max(a_{22}, a_{21}) = a_{21}$ and $\min
(a_{11}, a_{12}) = a_{11}$, according to the
previous lemma, strategy~1 wins whenever
\begin{eqnarray*}
\max_{z \neq N} \phi_2
(z) &=& \phi_2 (0) = a_{21} = (1 - 1/N) \max(a_{22},
a_{21}) + (1/N) a_{21}
\\
&< &(1 - 1/N) \min(a_{11}, a_{12}) + (1/N) a_{12}
\\
&= &(1 - 1/N) a_{11} + (1/N) a_{12} = \phi_1 (1)
= \min_{z
\neq0} \phi_1 (z),
\end{eqnarray*}
which is true whenever \eqref{eq:biased-case-0} holds.
\item\emph{Case} 3. When $\max(a_{22}, a_{21}) = a_{22}$ and $\min
(a_{11}, a_{12}) = a_{12}$, according to the
previous lemma, strategy~1 wins whenever
\begin{eqnarray*} \max_{z \neq N} \phi_2
(z) &=& \phi_2 (N - 1) = (1 - 1/N) a_{22} + (1/N)
a_{21}
\\
&=& (1 - 1/N) \max(a_{22}, a_{21}) + (1/N) a_{21}
\\
&<& (1 - 1/N) \min(a_{11}, a_{12}) + (1/N) a_{12} =
a_{12} = \phi_1 (N) = \min_{z \neq0}
\phi_1 (z),
\end{eqnarray*}
which is true whenever \eqref{eq:biased-case-0} holds.
\item\emph{Case} 4. When $\max(a_{22}, a_{21}) = a_{22}$ and $\min
(a_{11}, a_{12}) = a_{11}$, according to the
previous lemma, strategy~1 wins whenever
\begin{eqnarray*}\max_{z \neq N} \phi_2
(z) &=& \phi_2 (N - 1) = (1 - 1/N) a_{22} + (1/N)
a_{21}
\\
&=& (1 - 1/N) \max(a_{22}, a_{21}) + (1/N) a_{21}\\
& <&
(1 - 1/N) \min (a_{11}, a_{12}) + (1/N) a_{12}
\\
&=& (1 - 1/N) a_{11} + (1/N) a_{12} = \phi_1 (1)
= \min_{z
\neq0} \phi_1 (z),
\end{eqnarray*}
which is true whenever \eqref{eq:biased-case-0} holds.
\end{itemize}
This completes the proof of the lemma and the proof of Theorem~\ref{th:biased}.
\end{pf}


\section{Proof of Theorem~\texorpdfstring{\protect\ref{th:coex}}{5}}
\label{sec:coex}

The common background behind the proofs of the remaining three
theorems is the use of a block construction, though
the arguments required to indeed be able to apply this technique
strongly differ among these theorems.
The idea of the block construction is to couple a certain collection
of good events related to the process properly rescaled in space
and time with the set of open sites of oriented site percolation on
the directed graph $\mathcal H_1$ with vertex set
\[
H:= \bigl\{(z, n) \in\Z^d \times\Z_+ \dvtx z_1 +
z_2 + \cdots + z_d + n \mbox { is even} \bigr\}
\]
and in which there is an oriented edge
\begin{eqnarray}
 (z, n) \to\bigl(z', n'
\bigr) \nonumber\\
\eqntext{\mbox{if and only if }
z' = z \pm e_i \mbox{ for some } i = 1, 2, \ldots, d
\mbox{ and } n' = n + 1, }
\end{eqnarray}
where $e_i$ is the $i$th unit vector.
See the left-hand side of Figure~\ref{fig:graphs} for a picture in $d
= 1$.
For a definition of oriented site percolation, we refer to Durrett
\cite{durrett_1995} where the block construction is also reviewed
in detail and employed to study different spatial processes.
The existence of couplings between the spatial game and oriented
percolation relies, among other things, on the application of
Theorem~4.3 in~\cite{durrett_1995} which requires certain good events
to be measurable with respect to a so-called graphical
representation of the process.
Therefore, we need to construct the spatial game from a graphical
representation, though we will not use it explicitly except in the
last section.
To construct the process graphically, we first observe that, in view
of \eqref{eq:brown-hansell}, the maximum rate at which a
player gives birth over all possible configurations is given by
%
\begin{eqnarray}
\label{eq:max-birth} %
\max_{\eta}
\phi(0, \eta) \ind\bigl\{\phi(0, \eta) > 0 \bigr\} &=& \max_{\eta}
\max\bigl(0, \phi(0, \eta)\bigr)
\nonumber
\\[-8pt]
\\[-8pt]
\nonumber
&=& \max\Bigl(0, \max_{\eta} \phi(0, \eta)\Bigr) = \max\Bigl(0,
\max_{i,
j} a_{ij}\Bigr).
\end{eqnarray}
Similarly, the maximum rate at which a player dies is
%
\begin{eqnarray}
\label{eq:max-death} %
 \max_{\eta}
\bigl(- \phi(0, \eta) \ind\bigl\{\phi(0, \eta) < 0 \bigr\}\bigr)& =& \max
_{\eta} \max\bigl(0, - \phi(0, \eta)\bigr)
\nonumber
\\
&=& \max\Bigl(0, \max_{\eta} \bigl(- \phi(0, \eta)\bigr)\Bigr)\\
& =&
\max\Bigl(0, \max_{i, j} (- a_{ij})\Bigr). \nonumber
\end{eqnarray}
From \eqref{eq:max-birth}--\eqref{eq:max-death}, we deduce that the
maximum rate at which a player either gives birth or dies over
all the possible configurations is given by
\[
 \mathfrak m:= \max\Bigl(\max_{i, j}
a_{ij}, \max_{i, j} (- a_{ij})\Bigr) = \max
_{i, j} |a_{ij}|.
\]
The process is then constructed graphically from a collection of
independent Poisson processes with intensity $\mathfrak m$
by using the following well-known property: extracting points
independently with probability $p$ from a Poisson point process
with intensity $\mathfrak m$ results in a Poisson point process with
reduced intensity $p \times\mathfrak m$.
More precisely, for all $x \in\Z^d$ and $n > 0$:
\begin{itemize}
\item we let $T_n (x) =$ the $n$th arrival time of a Poisson process
with rate $\mathfrak m$;
\item we let $U_n (x) =$ a uniform random variable over the interval
$(0, \mathfrak m)$;
\item we let $V_n (x) =$ a uniform random variable over the interaction
neighborhood~$N_x$.
\end{itemize}
At the arrival times $T_n (x)$, we draw
%
\begin{equation}
\label{eq:coex-arrow} \mbox{an arrow } V_n (x) \to x \mbox{ with the label }
U_n (x)
\end{equation}
and say that this arrow is \emph{active} whenever
\[
U_n (x) < \bigl|\phi(x, \eta_{t-})\bigr| \qquad\mbox{where } t:=
T_n (x).
\]
Staring from any initial configuration, an argument due to Harris \cite
{harris_1972} implies that the process can be constructed going forward
in time by setting
\[
\eta_t (x) = \eta_t \bigl(V_n (x)\bigr):=
\cases{ %
 \eta_{t-} (x), & \quad $\mbox{when }
\phi(x, \eta_{t-}) > 0 \mbox{ and \eqref{eq:coex-arrow} is active},$
\vspace*{2pt}\cr
\eta_{t-} \bigl(V_n (x)\bigr), & \quad$\mbox{when }  \phi(x,
\eta_{t-}) < 0 \mbox { and \eqref{eq:coex-arrow} is active},$}
\]
where again $t:= T_n (x)$.
In case the arrow in \eqref{eq:coex-arrow} is not active, the update
is canceled.
In order to simplify a little bit some cumbersome expressions in the
proofs of the remaining three theorems, we identify from now on the spatial
game with the set of the type~1 players, which is a common approach to
study spin systems.
We now focus on the proof of our coexistence result.
The first step is to establish a strong form of survival of the type~1
players when
%
\begin{equation}\quad
\label{eq:coex-limit} (M, d) \neq(1, 1)\quad \mbox{and}\quad a_{12} = a_{21}
= 0 \quad\mbox{and}\quad c (M, d) a_{22} < a_{11} < - 1,
\end{equation}
which is done by comparing the spatial game $(\xi_t)$ with $a_{12} =
a_{21} = 0$ and one dependent oriented site percolation.
The reason for studying first the process under assumption \eqref
{eq:coex-limit} is to prevent extinction, that is, ensure a weak form
of survival,
of the set of type~1 players, which facilitates our proof of strong survival.
The full result is then deduced by using a perturbation argument.
To make the idea rigorous, we declare site $(z, n) \in H$ to be \emph{occupied} whenever
\[
\xi_{cn K} \cap B_2 (Kz, 3K/5) \neq\varnothing,
\]
where $c > 0$ is a constant, and $K$ a large integer that will be fixed
later and where $B_2 (x, r)$ is the Euclidean ball with center $x$ and
radius $r$.
Also we set
%
\begin{equation}
\label{eq:level} \mathbb X_n:= \bigl\{z \in\Z^d \dvtx (z,
n) \in H \mbox{ and is occupied} \bigr\}.
\end{equation}
%
In view of Theorem~4.3 in Durrett \cite{durrett_1995} and the fact
that the spatial game is translation invariant in space and time,
to prove that the process $\mathbb X_n$ dominates stochastically
supercritical oriented site percolation, the main step is to show
that the conditional probability
\[
P \bigl((e_1, 1) \mbox{ is occupied } | (0, 0) \mbox{ is occupied}
\bigr)
\]
can be made arbitrarily close to one by choosing $K$ sufficiently large.
To estimate this conditional probability, we start with a single
type~1 player at site 0 and keep track of a specific player of type~1
that moves to the target $K e_1$.
Let $\pi_1$ be the projection onto the first axis, and denote by
\[
r_t:= \max\bigl\{\pi_1 (x) \dvtx x \in
\xi_t \bigr\} \quad\mbox{and}\quad R_t:= \bigl\{x \in
\xi_t \dvtx \pi_1 (x) = r_t \bigr\}
\]
the first coordinate of the rightmost type~1 players and the set of the
rightmost type~1 players, respectively.
Since $a_{12} = a_{21} = 0$, this set is always nonempty.
In one dimension, it reduces to a singleton whereas in higher
dimensions it may have more sites.
In any case, we let $X_t$ be the position of one of the rightmost
type~1 players chosen uniformly at random among the ones who
are the closest to the first axis, and call this player the \emph{tagged player}.
The key to proving that the set of type~1 players spreads in the
direction of $e_1$ is given by the next lemma.

%
\begin{lemma}
\label{lem:coex-drift}
Assume \eqref{eq:coex-limit}. Then there exists $\mu> 0$ such that
\[
\lim_{h \to0} h^{-1} E
\bigl(\pi_1 (X_{t + h}) - \pi _1
(X_t) | \xi_t\bigr) \geq\mu\qquad\mbox{almost surely}.
\]
\end{lemma}

\begin{pf}
To begin with, we introduce the process
\[
L_t:= \inf\bigl\{\pi_1 (X_t - x) \dvtx x
\in\xi_t \mbox{ and } x \neq X_t \bigr\}.
\]
In words, the process $L_t$ keeps track of the distance along the first
axis between the tagged player and the second rightmost player
of type~1, which is also the length of a jump to the left of the
projection on the first axis~$\pi_1 (X_t)$ at the time the tagged player
changes her strategy.
We refer the reader to Figure~\ref{fig:coex-1} for a picture
describing the neighborhood of $X_t$.
To prove the lemma, we distinguish two cases depending on the value of
the process $L_t$.

\begin{figure}[b]

\includegraphics{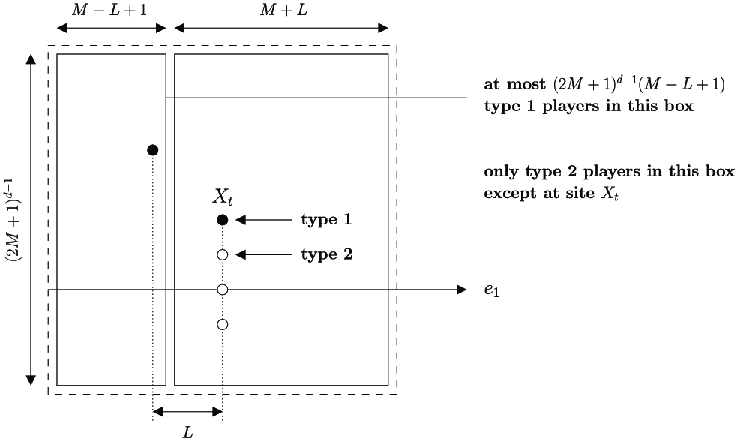}

\caption{Picture related to the proof of Lemma~\protect\ref{lem:coex-drift}.}
\label{fig:coex-1}
\end{figure}

\emph{Case} 1. Assume that $L_t = L \in[1, M]$.
Then the number of type~1 neighbors of the tagged player is bounded by
the number of vertices in the leftmost rectangle in the picture,
\[
f_1 (X_t, \xi_t) \leq N^{-1}
(2M + 1)^{d - 1} (M - L + 1).
\]
Therefore the rate at which the strategy at $X_t$ changes from~$1 \to
2$ is
\begin{eqnarray*}c (X_t, \xi_t) & \leq&
N^{-2} (- a_{11}) \max\bigl\{z (N - z) \dvtx z \leq(2M +
1)^{d - 1} (M - L + 1) \bigr\}
\\
& = & N^{-2} (- a_{11}) (2M + 1)^{d - 1} (M - L + 1)
\bigl((2M + 1)^{d - 1} (M + L) - 1\bigr).
 \end{eqnarray*}
Since such a transition causes the tagged player to jump~$L$ units to
the left or equivalently the first coordinate~$\pi_1 (X_t)$ of the
process~$X_t$ to decrease by the amount $L$, the previous inequality
also gives the following bound almost surely on the transition rate:
%
\begin{eqnarray}\qquad
\label{eq:jump-left} %
&&\lim_{h \to0}
h^{-1} P \bigl(\pi_1 (X_{t + h}) -
\pi_1 (X_t) = - L | \xi _t \mbox{ and }
L_t = L \in[1, M]\bigr)
\nonumber
\\[-8pt]
\\[-8pt]
\nonumber
&&\qquad\leq N^{-2} (- a_{11}) (2M + 1)^{d - 1} (M - L + 1)
\bigl((2M + 1)^{d - 1} (M + L) - 1\bigr).
\end{eqnarray}
In addition, each site $x$ occupied by a player of type~2 in the
neighborhood of $X_t$ has at least one neighbor of type~1, namely the tagged
player, therefore the rate at which the strategy at such a
neighbor~$x$ changes from~$2 \to1$ is at least equal to
\begin{eqnarray*} c (x, \xi_t) & \geq&
N^{-2} (- a_{22}) \min\bigl\{z (N - z) \dvtx z \neq0 \bigr\}
\\
& = & N^{-2} (- a_{22}) (N - 1) = N^{-2} (-
a_{22}) \bigl((2M + 1)^d - 2\bigr).
\end{eqnarray*}
Since such a transition causes the tagged player to jump to $x$ when
$\pi_1 (x) > \pi_1 (X_t)$ and since the number of such neighbors
of the tagged player is given by
\[
\card\bigl\{x \in N_{X_t} \dvtx \pi_1 (x) =
\pi_1 (X_t) + j \mbox{ and } x \notin\xi_t
\bigr\} = (2M + 1)^{d - 1}
\]
for all $j = 1, 2, \ldots, M$, we deduce that
%
\begin{eqnarray}
\label{eq:jump-right} %
&&\lim_{h \to0}
h^{-1} P \bigl(\pi_1 (X_{t + h}) -
\pi_1 (X_t) = j | \xi_t\bigr)
\nonumber
\\[-8pt]
\\[-8pt]
\nonumber
&&\qquad\geq N^{-2} (- a_{22}) \bigl((2M + 1)^d - 2
\bigr) (2M + 1)^{d - 1}
\end{eqnarray}
almost surely for all $j = 1, 2, \ldots, M$.
Using as previously mentioned that \eqref{eq:jump-left} is the only
transition that can decrease the first coordinate of the tagged player
and summing the transition rates in \eqref{eq:jump-right} over all the
possible values of $j$, we deduce that, almost surely,
\begin{eqnarray*}&& \lim_{h \to0} h^{-1} E
\bigl(\pi_1 (X_{t + h}) - \pi_1
(X_t) | \xi_t \mbox{ and } L_t = L \in[1, M]
\bigr)
\\
&&\qquad\geq(- L) \lim_{h \to0} h^{-1} P \bigl(
\pi_1 (X_{t + h}) - \pi _1 (X_t) = -
L | \xi_t \mbox{ and } L_t = L \in[1, M]\bigr)
\\
&&\qquad\quad{}+ \sum_{j = 1}^M j \lim
_{h \to0} h^{-1} P \bigl(\pi_1
(X_{t + h}) - \pi_1 (X_t) = j |
\xi_t\bigr)
\\
&&\qquad\geq(- L) N^{-2} (- a_{11}) (2M + 1)^{d - 1} (M - L +
1) \bigl((2M + 1)^{d - 1} (M + L) - 1\bigr)
\\
&&\qquad\quad{}+ \sum_{j =
1}^M j N^{-2} (-
a_{22}) \bigl((2M + 1)^d - 2\bigr) (2M + 1)^{d
- 1}.
\end{eqnarray*}
Expanding and simplifying the right-hand side gives
\begin{eqnarray*}&&\lim_{h \to0} h^{-1} E
\bigl(\pi_1 (X_{t + h}) - \pi_1
(X_t) | \xi_t \mbox{ and } L_t = L \in[1, M]
\bigr)
\\
&&\qquad\geq N^{-2} (2M + 1)^{d - 1} \bigl[a_{11} L (M - L +
1) \bigl((2M + 1)^{d - 1} (M + L) - 1\bigr)
\\
&&\hspace*{124pt}\qquad{}- a_{22} (1/2) M (M + 1) \bigl((2M + 1)^d - 2\bigr)
\bigr].
\end{eqnarray*}
Using that $L (M - L + 1) \leq(1/4)(M + 1)^2$ and $M + L \leq2M$,
we obtain
\begin{eqnarray*}
&&\lim_{h \to0} h^{-1} E
\bigl(\pi_1 (X_{t + h}) - \pi_1
(X_t) | \xi_t \mbox{ and } L_t = L \in[1, M]
\bigr) \\
&&\qquad\geq N^{-2} (2M + 1)^{d - 1} \bigl[a_{11} (1/4)
(M + 1)^2 \bigl(2M (2M + 1)^{d - 1} - 1\bigr)
\\
&&\hspace*{124pt}{}- a_{22} (1/2) M (M + 1) \bigl((2M + 1)^d - 2\bigr) \bigr]
\\
&&\qquad= N^{-2} (2M + 1)^{d - 1} (1/4) (M + 1)
\\
&&\qquad\quad{}\times\bigl(a_{11} (M + 1) \bigl(2M (2M + 1)^{d - 1} - 1\bigr)
- a_{22} (2M) \bigl((2M + 1)^d - 2\bigr)\bigr)
\\
&&\qquad= N^{-2} (2M + 1)^{d - 1} (1/4) (M + 1)^2 \bigl(2M
(2M + 1)^{d - 1} - 1\bigr) \\
&&\qquad\quad{}\times\bigl(a_{11} - c (M, d)
a_{22}\bigr) > 0
\end{eqnarray*}
almost surely whenever \eqref{eq:coex-limit} holds.

\emph{Case} 2. Assume that $L_t = L \notin[1, M]$. Then we
have the following alternative:
\begin{itemize}
\item$L = 0$, and then there are at least two vertices in the set
$R_t$.
\item$L > M$, and then the tagged player has only type~2 players in
her neighborhood and therefore changes her strategy at rate zero.
\end{itemize}
In either case, $\pi_1 (X_t)$ cannot decrease, so \eqref{eq:jump-right}
implies that
\begin{eqnarray*} &&\lim_{h \to0} h^{-1} E
\bigl(\pi_1 (X_{t + h}) - \pi_1
(X_t) | \xi_t \mbox{ and } L_t = L \notin[1,
M]\bigr)
\\
&&\qquad\geq\sum_{j = 1}^M j \lim
_{h \to0} h^{-1} P \bigl(\pi_1
(X_{t + h}) - \pi_1 (X_t) = j |
\xi_t\bigr)
\\
&&\qquad\geq N^{-2} (2M + 1)^{d - 1} (1/4) (M + 1) (- a_{22})
(2M) \bigl((2M + 1)^d - 2\bigr) > 0
\end{eqnarray*}
almost surely whenever \eqref{eq:coex-limit} holds.
This completes the proof.
\end{pf}

The previous lemma is similar to pages 1247--1248 in \cite
{neuhauser_pacala_1999}.
There, the authors conclude that we can bring a particle---the tagged
player in our case---close to the target $Ke_1$.
This is obvious in one dimension.
In higher dimensions, the idea is to use the lemma to increase the
first coordinate of the tagged player up to $K$ and then apply the lemma
again along each of the other $d - 1$ axes to bring the tagged player
close to the target.
However, since we do not have control on the position of the tagged
player in the direction orthogonal to $e_1$ while moving along the first
axis, the conclusion is not obvious.
To prove that we can bring a type~1 player close to the target in
higher dimensions, we look instead at the Euclidean distance between
the target and the
type~1 player the closest to the target.
We now call $X_t$ the position of one of the type~1 players chosen
uniformly at random among the ones who are the closest to $Ke_1$,
called again
the tagged player, and prove that
%
\begin{equation}
\label{eq:coex-trig} %
 \qquad\lim_{h \to0} \sup
_{x \notin B_2 (K e_1, K/5)} h^{-1} E (D_{t + h} - D_t |
X_t = x) \leq - \mu\qquad\mbox {for some $\mu> 0$},
\end{equation}
where $D_t:= \dist(X_t, K e_1) =$ Euclidean distance between $X_t$
and $K e_1$ and where $\Omega_K$ is a set of configurations in which
the tagged player is far from the target,
\[
\Omega_K:= \bigl\{\eta\subset\Z^d \dvtx \eta\cap
B_2 (K e_1, K/5) = \varnothing\bigr\} \qquad\mbox{for $K$
large}.
\]
Although our proof relies on basic trigonometry, the algebra is
somewhat messy, so we only prove the result in the two-dimensional nearest
neighbor case.
Hopefully, the next lemma will convince the reader that, even if the
players are located on a square lattice, the type~1 players spread
not only along each axis but also along any arbitrary direction
provided condition \eqref{eq:coex-limit} holds.

%
\begin{lemma}
\label{lem:coex-trigo}
Assume that $(M, d) = (1, 2)$, and \eqref{eq:coex-limit} holds. Then,
\eqref{eq:coex-trig} holds for all $K$ large.
\end{lemma}

\begin{pf}
The proof is based on the construction given in Figure~\ref
{fig:coex-2}. Let:
\begin{eqnarray*}
 C &:= & \mbox{the circle with center $K
e_1$ going through $X_t$},
\\
\Delta&:= & \mbox{the tangent line to the circle $C$ going through
$X_t$},
\\
\Gamma&:= & \mbox{the straight line parallel to the tangent $\Delta$ going
through $K e_1$}.
\end{eqnarray*}
The first ingredient is to observe that, on the event that $\xi_t \in
\Omega_K$,
%
\begin{eqnarray}
\label{eq:approx} %
&&\lim_{h \to0} \sup
_{x \notin B_2 (K e_1, K/5)} h^{-1} E (D_{t + h} - D_t |
X_t = x)
\nonumber
\\[-8pt]
\\[-8pt]
\nonumber
&&\qquad\approx\lim_{h \to0} \sup_{x \notin B_2 (K e_1, K/5)}
h^{-1} E \bigl(\dist(X_{t + h}, \Gamma) - \dist(X_t,
\Gamma) | X_t = x\bigr)
\end{eqnarray}
when the parameter $K$ is large, so it suffices to prove the result for
the right-hand side.
To estimate the drift, note that the straight line $\Delta$ divides
the neighborhood of $X_t$ into two sets of four vertices:
as indicated on the left-hand side of the figure, we denote by $y_i$
the four vertices the closest to the target $K e_1$ and we denote by $x_i$
the other four vertices in such a way that
\[
l_i:= \dist(y_i, \Delta) = \dist(x_i,
\Delta) \qquad\mbox{for } i = 1, 2, 3, 4.
\]

%
\begin{figure}[b]

\includegraphics{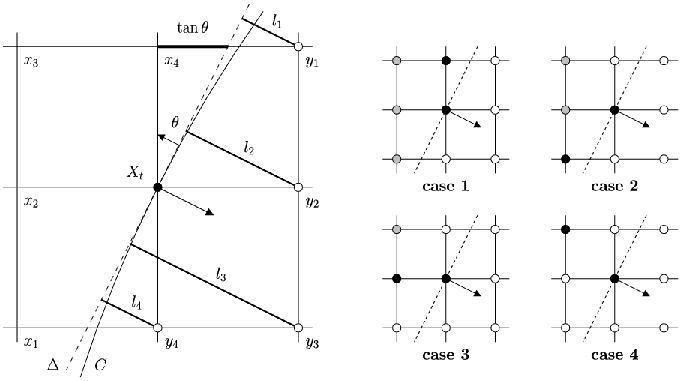}

\caption{Picture related to the proof of Lemma~\protect\ref{lem:coex-trigo}.}
\label{fig:coex-2}
\end{figure}

\noindent Defining the angle $\theta$ as in the picture, some basic trigonometry
shows that
%
\begin{eqnarray}
\label{eq:trigo} %
l_1 & = & (1 -
\tan\theta) \cos\theta,\qquad l_2  = \cos \theta,
\nonumber
\\[-8pt]
\\[-8pt]
\nonumber
l_3 & = & (1 + \tan\theta) \cos\theta,\qquad  l_4  =  \tan
\theta\cos\theta.
\end{eqnarray}
We may assume that $0 \leq\theta\leq\pi/4$ and so $\tan\theta\in
[0, 1]$ since any other configuration can be deduced from a rotation of
this configuration.
Note that all four players at sites $y_i$ must follow strategy~2,
which gives $2^4 = 16$ possible configurations in the neighborhood of $X_t$.
To find a bound for the drift, we only distinguish four types of
configurations (see Figure \ref{fig:coex-2}).
\begin{itemize}
\item\emph{Case} 1. Assume that $x_4 \in\xi_t$. Then
\begin{eqnarray*}
&&\lim_{h \to0} \sup
_{x \notin B_2 (K e_1, K/5)} h^{-1} E \bigl(\dist(X_{t + h}, \Gamma)
- \dist(X_t, \Gamma) | X_t = x\bigr)
\\
&&\qquad\leq(- a_{11}) (16/64) l_4 - (- a_{22})
\\
&&\quad\qquad{}\times\bigl((12/64) l_1 + (12/64) l_2 + (7/64) l_3
+ (7/64) l_4\bigr).
\end{eqnarray*}
Using \eqref{eq:trigo}, we deduce that
\begin{eqnarray*}
&&\lim_{h \to0} \sup
_{x \notin B_2 (K e_1, K/5)} h^{-1} E \bigl(\dist(X_{t + h}, \Gamma)
- \dist(X_t, \Gamma) | X_t = x\bigr)
\\
&&\qquad\leq(1/64) \cos\theta\bigl(- a_{11} (16 \tan\theta) + a_{22}
(31 + 2 \tan\theta)\bigr) < 0
\end{eqnarray*}
whenever $a_{11} > (33 / 16) a_{22}$ which holds if $a_{11} > c (1,
2) a_{22} = (7 / 5) a_{22}$.
\item\emph{Case} 2. Assume that $x_4 \notin\xi_t$ and $x_1 \in
\xi_t$. Then
\begin{eqnarray*} &&\lim_{h \to0} \sup
_{x \notin B_2 (K e_1, K/5)} h^{-1} E \bigl(\dist(X_{t + h}, \Gamma)
- \dist(X_t, \Gamma) | X_t = x\bigr)
\\
&&\qquad\leq(1/64) \cos\theta\bigl(- a_{11} \times15 l_1 +
a_{22} \times(7 l_1 + 7 l_2 + 7
l_3 + 12 l_4)\bigr)
\\
&&\qquad= (1/64) \cos\theta\bigl(- a_{11} (15 - 15 \tan\theta) +
a_{22} (21 + 12 \tan\theta)\bigr) < 0
\end{eqnarray*}
whenever $a_{11} > c (1, 2) a_{22} = (7 / 5) a_{22}$.
\item\emph{Case} 3. Assume that $x_4, x_1 \notin\xi_t$ and $x_2
\in\xi
_t$. Then
\begin{eqnarray*} &&\lim_{h \to0} \sup
_{x \notin B_2 (K e_1, K/5)} h^{-1} E \bigl(\dist(X_{t + h}, \Gamma)
- \dist(X_t, \Gamma) | X_t = x\bigr)
\\
&&\qquad\leq(1/64) \cos\theta\bigl(- a_{11} \times12 l_2 +
a_{22} \times(7 l_1 + 7 l_2 + 7
l_3 + 12 l_4)\bigr)
\\
&&\qquad= (1/64) \cos\theta\bigl(- 12 a_{11} + a_{22} (21 + 12 \tan
\theta)\bigr) < 0
\end{eqnarray*}
whenever $a_{11} > (7 / 4) a_{22}$ which holds if $a_{11} > c (1, 2)
a_{22} = (7 / 5) a_{22}$.
\item\emph{Case} 4. Assume that $x_4, x_1, x_2 \notin\xi_t$ and $x_3
\in\xi_t$. Then
\begin{eqnarray*} &&\lim_{h \to0} \sup
_{x \notin B_2 (K e_1, K/5)} h^{-1} E \bigl(\dist(X_{t + h}, \Gamma)
- \dist(X_t, \Gamma) | X_t = x\bigr)
\\
&&\qquad\leq(1/64) \cos\theta\bigl(- a_{11} \times7 l_2 +
a_{22} \times (7 l_1 + 7 l_2 + 7
l_3 + 7 l_4)\bigr)
\\
&&\qquad= (1/64) \cos\theta\bigl(- a_{11} (7 + 7 \tan\theta) +
a_{22} (21 + 7 \tan\theta)\bigr) < 0
\end{eqnarray*}
whenever $a_{11} > 2 a_{22}$ which holds if $a_{11} > c (1, 2) a_{22}
= (7 / 5) a_{22}$.
\end{itemize}
This, together with the approximation \eqref{eq:approx}, implies the lemma.
\end{pf}

We now use Lemmas \ref{lem:coex-drift} and \ref{lem:coex-trigo} to
prove that, with probability close to one when $K$ is large,
the tagged player is located in a certain Euclidean ball with center
$Ke_1$ at a deterministic time proportional to $K$.
This is done in Lemmas \ref{lem:coex-target-1}--\ref
{lem:coex-target-3} below where we successively prove that the tagged player
hits a subset of the target region in a short time, does not leave a
certain larger ball centered at zero, and stays in the target
region for a long time.
The second step is needed to ensure that the events under
consideration are measurable with respect to the graphical representation
in a finite space--time box, which is a key to obtaining a coupling
between the process and oriented percolation with a finite
range of dependence.
For every positive integer $K$, define
\[
\tau_K:= \inf\bigl\{t \dvtx X_t \in B_2
(K e_1, 2 K/5) \bigr\} = \inf\{ t \dvtx D_t < 2 K/5 \}.
\]
%

%
\begin{lemma}
\label{lem:coex-target-1}
Assume \eqref{eq:coex-limit}.
There exist $c, C_1 < \infty$ and $\gamma_1 > 0$ such that
\[
P \bigl(\tau_K \geq c K | X_0 \in B_2 (0,
3 K/5)\bigr) \leq C_1 \exp (- \gamma_1 K)\qquad \mbox{for all
$K$ large}.
\]
\end{lemma}

\begin{pf}
According to Lemmas~\ref{lem:coex-drift} and~\ref{lem:coex-trigo},
%
\begin{equation}
\label{eq:coex-target-1} %
\qquad\lim_{h \to0} \sup
_{x \notin B_2 (K e_1, K/5)} h^{-1} E (D_{t + h} - D_t |
X_t = x) \leq- \mu\qquad\mbox {for some } \mu> 0
\end{equation}
from which it follows that
\[
E D_t \leq D_0 - \mu t \qquad\mbox{for all } t <
\mu^{-1} (D_0 - K/5).
\]
Using in addition that the number of jumps of the process $(D_t)$
dominates stochastically the number of jumps of a Poisson process with
positive intensity, large deviation estimates for the Poisson
distribution imply that
\begin{eqnarray*} &&P \bigl(\tau_K \geq2
\mu^{-1} K | X_0 \in B_2 (0, 3 K/5)\bigr)
\\
&&\qquad\leq P \bigl(D_t \geq K/5 \mbox{ for all } t < 2 \mu^{-1} K
| D_0 \leq7 K/5\bigr) \leq C_1 \exp(-
\gamma_1 K)
\end{eqnarray*}
for suitable constants $C_1 < \infty$ and $\gamma_1 > 0$ and all $K$
sufficiently large.
\end{pf}

%
\begin{lemma}
\label{lem:coex-target-2}
Assume \eqref{eq:coex-limit}.
There exist $C_2 < \infty$ and $\gamma_2 > 0$ such that
\[
P \bigl(D_t \geq2K \mbox{ for some } t \in(0, cK) | X_0
\in B_2 (0, 3 K/5)\bigr) \leq C_2 \exp(-
\gamma_2 K)
\]
for all $K$ sufficiently large, and where $c$ is as in Lemma~\ref
{lem:coex-target-1}.
\end{lemma}

\begin{pf}
First, we introduce the stopping times
\[
\sigma_K:= \inf\{t \dvtx D_t \geq2K \}\quad \mbox{and}\quad
T_K:= \inf(\tau_K, \sigma_K)
\]
and the process stopped at time $T_K$
\[
Z_t:= \exp(a D_t) \ind\{t < T_K \} +
\exp(a D_{T_K}) \ind\{ t \geq T_K \}.
\]
As in Lemma~\ref{lem:1D-drift}, the key to the proof is to find a
constant $a > 0$ such that the process $(Z_t)$ is a supermartingale
with respect to the
natural filtration of the process $(\xi_t)$ and then apply the optimal
stopping theorem.
To prove the existence of such a constant, we introduce
\[
 \Phi(a):= \lim_{h \to0}
h^{-1} E \bigl(Z_{t +
h} (a) - Z_t (a) |
\xi_t\bigr)
\]
and observe that, for all $t < T_K$,
\begin{eqnarray*}
\Phi(a) &=& \sum_{x \in\Z^d}
\bigl(\exp\bigl(a \dist (X_t + x, K e_1)\bigr)
- \exp\bigl(a \dist (X_t, K e_1)\bigr)\bigr) \\
&&\hspace*{18pt}{}\times\lim
_{h \to0} h^{-1} P (X_{t + h} - X_t =
x | \xi_t).
\end{eqnarray*}
Recalling \eqref{eq:coex-target-1} and using that $D_t \geq K/5$ for
all $t < T_K$, we deduce that
\begin{eqnarray*} \Phi' (0) & = & \sum
_{x \in\Z^d} \bigl(\dist(X_t + x, K e_1) -
\dist(X_t, K e_1)\bigr) \lim_{h \to0}
h^{-1} P (X_{t + h} - X_t = x | \xi_t)
\\
& = & \lim_{h \to0} h^{-1} E (D_{t + h} -
D_t | \xi_t) \leq- \mu< 0\qquad \mbox{almost surely}.
\end{eqnarray*}
Since in addition $\Phi(0) = 0$, there exists $a_0 > 0$ fixed from now
on such that
\[
\Phi(a_0):= \lim
_{h \to0} h^{-1} E (Z_{t
+ h} - Z_t
| \xi_t) \leq0 \qquad\mbox{almost surely},
\]
which implies that $(Z_t)$ is a supermartingale for this value of $a$.
Since the stopping time $T_K$ is almost surely finite, the optimal
stopping theorem further implies that
\begin{eqnarray*}
E Z_{T_K} & \leq& E Z_0 =
E \exp(a_0 X_0) \leq\exp\bigl(a_0 (K + 3
K/5)\bigr) = \exp\bigl(a_0 (8 K/5)\bigr),
\\
E Z_{T_K} & \geq& \exp(2a_0 K) P (\sigma_K <
\tau_K) + \exp\bigl(a_0 (2 K/5 - M)\bigr) P (
\tau_K < \sigma_K)
\\
& \geq & \exp(2a_0 K) P (\sigma_K <
\tau_K) + \exp\bigl(a_0 (2 K/5 - M)\bigr) \bigl(1 - P (
\sigma_K < \tau_K)\bigr)
\end{eqnarray*}
from which we deduce that
\begin{eqnarray*} P (\sigma_K < \tau_K)
& \leq& \bigl[\exp\bigl(a_0 (8 K/5)\bigr) - \exp\bigl(a_0 (2
K/5 - M)\bigr)\bigr]
\\
&&{} \times\bigl[\exp(2a_0 K) - \exp\bigl(2a_0 (2 K/5 - M)
\bigr)\bigr]^{-1}
\\
& = & \bigl[\exp\bigl(a_0 (6 K/5 + M) - 1\bigr)\bigr] \bigl[\exp
\bigl(a_0 (8 K/5 + M) - 1\bigr)\bigr]^{-1}
\\
& \leq& \exp(- 2a_0 K/5).
\end{eqnarray*}
Since the probability that the number of jumps of the tagged player by
time $cK$ exceeds a certain multiple of $K$ also has exponential decay,
the result follows.
\end{pf}

%
\begin{lemma}
\label{lem:coex-target-3}
Assume \eqref{eq:coex-limit}.
There exist $C_3 < \infty$ and $\gamma_3 > 0$ such that
\[
P \bigl(D_t \geq3 K/5 \mbox{ for some } t \in(\tau_K,
cK) | \tau_K < cK\bigr) \leq C_3 \exp(-
\gamma_3 K)
\]
for all $K$ sufficiently large.
\end{lemma}

\begin{pf}
The result directly follows by observing that
\begin{eqnarray*} && P \bigl(D_t \geq3 K/5 \mbox{ for some }
t \in(\tau_K, cK) | \tau _K < cK\bigr)
\\
&&\qquad\leq P \bigl(D_t \geq3 K/5 \mbox { for some } t \in(0, cK) |
D_0 < 2 K/5\bigr)
\end{eqnarray*}
and by following the argument of the proof of Lemma~\ref
{lem:coex-target-2} but using
\[
\tau'_K:= \inf\{t \dvtx D_t < K/5 \}\quad
\mbox{and}\quad \sigma '_K:= \inf\{t \dvtx D_t
\geq3 K/5 \}
\]
in place of the stopping times $\tau_K$ and $\sigma_K$.
\end{pf}

With Lemmas \ref{lem:coex-target-1}--\ref{lem:coex-target-3}, we are
now ready to couple the process properly rescaled in space and time
with oriented site percolation.
Denote by $\mathbb W_n^{1 - \ep}$ the set of wet sites at level $n$ in
a one dependent oriented site percolation process
on $\mathcal H_1$ in which sites are open with probability $1 - \ep$.
Recall that a site is said to be wet if it can be reached from level
zero by a path of open sites.

%
\begin{lemma}
\label{lem:coex-perco}
Assume \eqref{eq:coex-limit}, and let $\ep> 0$.
Then, for all $K$ sufficiently large, the process can be coupled with
oriented site percolation in such a way that
\[
\mathbb W_n^{1 - \ep} \subset\mathbb X_n \qquad\mbox{for
all } n \geq 0 \mbox{ whenever } \mathbb X_0 = \mathbb
W_0^{1 - \ep}.
\]
\end{lemma}

\begin{pf}
Let $\Omega(z, n)$ denote the event that site $(z, n) \in H$ is
occupied, that is,
\[
\xi_{cn K} \cap B_2 (Kz, 3K/5) \neq\varnothing.
\]
Lemmas~\ref{lem:coex-target-1} and~\ref{lem:coex-target-3} imply the
existence of a collection of events $G (z, n)$ measurable with respect
to the graphical representation of the process such that:
\begin{longlist}[(1)]
\item[(1)] for all $K$ sufficiently large, $P (G (z, n)) \geq1 - \ep
$, and such that
\item[(2)] we have the inclusions of events
\[
G (z, n) \cap\Omega(z, n) \subset\Omega(z \pm e_i, n + 1) \qquad\mbox{for
all } i = 1, 2, \ldots, d.
\]
\end{longlist}
In addition, Lemma~\ref{lem:coex-target-2} implies that these events
can be chosen so that
\begin{longlist}[(1)]
\item[(3)] $G (z, n)$ is measurable with respect to the graphical
representation in
\[
B_2 (Kz, 2K) \times\bigl[cnK, c (n + 1)K\bigr] = (Kz, cnK) +
B_2 (0, 2K) \times[0, cK].
\]
\end{longlist}
These are the assumptions of Theorem~4.3 in Durrett \cite
{durrett_1995}, from which the existence of the coupling between the
two processes directly follows.
\end{pf}

In the next lemma, which recalls the statement of Theorem~\ref
{th:coex}, we return to the process with general payoffs.
The proof relies on the previous lemma, the symmetry of the evolution
rules of the spatial game and a perturbation argument.

%
\begin{lemma}
\label{lem:coex}
For all $a_{12}$ and $a_{21}$ there exists $m > 0$ such that
coexistence occurs when
\[
c (M, d) a_{22} < a_{11} < - m \quad\mbox{and}\quad c (M, d)
a_{11} < a_{22} < - m.
\]
\end{lemma}

\begin{pf}
First, we fix $\ep< (1/2)(1 - p_c)$ positive where $p_c < 1$ is the
critical value of the oriented site percolation process introduced above.
To prove that both strategies can survive simultaneously, we extend
our previous definition of occupied site by calling $(z, n) \in H$ a
\emph{good} site whenever
\[
x \in\eta_{cn K} \quad\mbox{and}\quad y \notin\eta_{cn K} \qquad\mbox{for some }
x, y \in B_2 (K z, 3 K/5).
\]
Denote by $\mathbb Y_n$ the set of good sites at level $n$.
The symmetry of the evolution rules implies that the conclusion of
Lemma~\ref{lem:coex-perco} holds for $\mathbb Y_n$ provided that
%
\begin{eqnarray}
\label{eq:coex-region-1} a_{12} &=& a_{21} = 0 \quad\mbox{and}\quad c (M, d)
a_{22} < a_{11} < - 1\quad \mbox{and}
\nonumber
\\[-8pt]
\\[-8pt]
\nonumber
 c (M, d) a_{11}& <&
a_{22} < - 1.
\end{eqnarray}
Even though (weak) survival of both strategies when $a_{12} = a_{21} =
0$ is in fact trivial since in this case a player isolated from players
of her own type
cannot change her strategy, we point out that the coupling with
oriented site percolation is needed to obtain the full coexistence region.
Indeed, the parameter $K$ being fixed such that the process dominates
one dependent oriented site percolation with parameter $1 - \ep$, the
continuity of the transition rates
with respect to the payoffs implies the existence of a small $\rho=
\rho(K) > 0$ and a coupling of the processes such that
%
\begin{equation}
\label{eq:coex-coupling} \mathbb W_n^{1 - 2 \ep} \subset\mathbb
Y_n\qquad \mbox{for all } n \geq0 \mbox{ whenever } \mathbb Y_0 =
\mathbb W_0^{1 - 2 \ep}
\end{equation}
in a perturbation of the parameter region \eqref{eq:coex-region-1}
given by
%
\begin{eqnarray}
\label{eq:coex-region-2} - \rho&<& a_{12}, a_{21} < \rho\quad\mbox{and}\quad c (M,
d) a_{22} < a_{11} < - 1 \quad\mbox{and}
\nonumber
\\[-8pt]
\\[-8pt]
\nonumber
 c (M, d) a_{11}
&<& a_{22} < - 1.
\end{eqnarray}
In particular, letting $f \dvtx 2^H = \mbox{power set of } H \to\{0, 1 \}
$ be defined by
%
\begin{equation}
\label{eq:coupling} f \bigl(\{\mathbb W_n \dvtx n \geq0 \}\bigr):= \ind
\bigl\{\card(n \dvtx z \in \mathbb W_n) = \infty\mbox{ for all } z \in
\Z^d \bigr\}
\end{equation}
and using \eqref{eq:coex-coupling} and the monotonicity of $f$, we
obtain that for all $(x, t) \in\Z^d \times\R_+$,
\begin{eqnarray*}&& P (x \in\eta_{s_1} \mbox{ and } x \notin
\eta_{s_2} \mbox{ for some } s_1, s_2 > t)
\\
&&\qquad\geq P \bigl(\card(n \dvtx z \in\mathbb Y_n) = \infty\mbox{ for all }
z \in \Z^d\bigr) = E f \bigl(\{\mathbb Y_n \dvtx n \geq0
\}\bigr)
\\
&&\qquad\geq E f \bigl(\bigl\{\mathbb W_n^{1 - 2 \ep} \dvtx n \geq0 \bigr
\}\bigr) = P \bigl(\card\bigl(n \dvtx z \in\mathbb W_n^{1 - 2 \ep}
\bigr) = \infty\mbox{ for all } z \in\Z ^d\bigr) = 1
\end{eqnarray*}
since infinitely many sites are wet at level zero and $1 - 2 \ep> p_c$.
Note also that the first inequality follows from the fact that if a
site is good, then the corresponding space--time region obtained through
rescaling contains
both strategies, so the probability that any given vertex $x$ in this
region changes its strategy after one time unit is bounded from below
by a positive constant.
This proves coexistence of both strategies in the parameter region
\eqref{eq:coex-region-2}.
To deal with the general case when both payoffs $a_{12}$~and~$a_{21}$
are arbitrary, we fix a sufficiently large $m > 0$ such that
\[
a_{12} \in(- m \rho, m \rho) \quad\mbox{and}\quad a_{21} \in(- m \rho,
m \rho).
\]
Since the long-term behavior remains unchanged by speeding up time by
$m$, that is, multiplying all the payoffs by the same factor $m$, we
obtain coexistence in the parameter region
\[
c (M, d) a_{22} < a_{11} < - m \quad\mbox{and}\quad c (M, d)
a_{11} < a_{22} < - m.
\]
This proves the lemma and Theorem~\ref{th:coex}.
\end{pf}



\section{Proof of Theorem~\texorpdfstring{\protect\ref{th:richardson}}{2}}
\label{sec:richardson}

This section is devoted to the proof of Theorem~\ref{th:richardson}.
We first prove that, under the assumption of the theorem, strategy~1 survives.
The key to obtaining this partial result is to observe that, when the
first payoff $a_{11} = 1$ while the other payoffs are equal to zero,
the spatial game starting from suitable initial configurations
dominates stochastically a Richardson model \cite{richardson_1973}.
To also prove extinction of strategy~2, we use an idea of Lanchier
\cite{lanchier_2012} that extends from discrete-time to continuous-time
processes a result of Durrett \cite{durrett_1992}
which states that sites which are not wet do not percolate for
oriented site percolation models in which sites are open with
probability close to one.
Throughout this section, to shorten a little bit the expression of
certain events, we let
\[
B_r:= [-r, r]^d \qquad\mbox{for all } r > 0.
\]
The spatial boxes involved in the block construction in both this
section and the next section are translations of these boxes for an
appropriate radius $r$.

%
\begin{lemma}
\label{lem:richardson-block}
Let $\ep> 0$ and $a_{11} = 1$.
Then there exist $K, c, \rho> 0$ such that
\[
P \bigl(\eta_t \not\supset B_{2K} \mbox{ for some } t
\in(cK, 2 cK) | \eta_0 \supset B_K\bigr) \leq\ep
\qquad\mbox{for all } \bar a_{11} \in(- \rho, \rho)^3.
\]
\end{lemma}

\begin{pf}
We introduce the following auxiliary processes:
\begin{itemize}
\item the spatial game $\xi_t$ with payoffs $a_{11} = 1$ and $\bar
a_{11} = 0$ and
\item the Richardson model $\zeta_t$ with flip rate
\[
c_{\mathrm{RM}} (x, \zeta) = \beta^2 f_1 (x, \zeta)
\ind\{x \notin \zeta\}\qquad \mbox{where } \beta:= \bigl((2M + 1)^d - 1
\bigr)^{-1}.
\]
\end{itemize}
In the process $\xi_t$ all type~2 players have a zero payoff while each
type~1 player with at least one type~1 neighbor has a payoff equal to
$\beta$ from which it follows that
\begin{eqnarray*} c_{\mathrm{SG}} (x, \xi) & = & 0\qquad \mbox{if }
 x \in\xi,
\\
c_{\mathrm{SG}} (x, \xi) & \geq& \beta^2 \qquad \mbox{if }  x \notin\xi
\mbox { and } f_1 (y, \xi) \neq0 \mbox{ for some } y \in\xi\cap
N_x.
\end{eqnarray*}
Since in addition the property that each type~1 player has at least one
type~1 neighbor is preserved by the dynamics of the processes, we
deduce the existence of a coupling $(\zeta, \xi)$ such that
%
\begin{equation}
\label{eq:richardson-block-1} P (\zeta_t \subset\xi_t |
\zeta_0 = \xi_0 = B_K) = 1.
\end{equation}
In other respects, the shape theorem \cite{richardson_1973} for the
Richardson model implies the existence of a positive constant $c > 0$
fixed from now on such that
\begin{eqnarray*}
&&P \bigl(\zeta_t \not\supset
B_{2K} \mbox{ for some } t \in(cK, 2 cK) | \zeta_0 \supset
B_K\bigr)\\
&&\qquad = P (\zeta _{cK} \not\supset
B_{2K} | \zeta_0 \supset B_K)
\\
&&\qquad\leq P \bigl(\zeta_{cK} \not\supset B_{2K} |
\zeta_0 = \{0 \}\bigr) \leq \ep/ 2\qquad \mbox{for all $K$ large},
\end{eqnarray*}
where the equality between the first two lines holds because infected
sites in the Richardson model do not recover.
In view of \eqref{eq:richardson-block-1}, the same holds for the
spatial game, that is,
%
\begin{equation}
\label{eq:richardson-block-2} P \bigl(\xi_t \not\supset B_{2K}
\mbox{ for some } t \in(cK, 2 cK) | \xi_0 \supset B_K
\bigr) \leq\ep/ 2
\end{equation}
for all $K$ large.
Now, we fix $K$ such that \eqref{eq:richardson-block-2} holds.
The scale parameter $K$ and the constant $c$ being fixed, the
continuity of the transition rates of the spatial game with respect to
the payoffs implies the existence of a small
constant $\rho> 0$ and a coupling $(\eta, \xi)$ such that
%
\begin{equation}
\label{eq:richardson-block-3} P (\eta_t \cap B_{2K} \neq
\xi_t \cap B_{2K} \mbox{ for some } t \leq2 cK |
\eta_0 = \xi_0) \leq\ep/ 2
\end{equation}
whenever $\bar a_{11} \in(- \rho, \rho)^3$.
Combining \eqref{eq:richardson-block-2} and \eqref
{eq:richardson-block-3} gives
\begin{eqnarray*} &&P \bigl(\eta_t \not\supset
B_{2K} \mbox{ for some } t \in(cK, 2 cK) | \eta_0 \supset
B_K\bigr)
\\
&&\qquad\leq P \bigl(\xi_t \not\supset B_{2K} \mbox{ for some } t
\in(cK, 2 cK) | \xi_0 \supset B_K\bigr)
\\
&&\qquad\quad{}+ P (\eta_t \cap B_{2K} \neq\xi_t \cap
B_{2K} \mbox{ for some } t \leq2 cK | \eta_0 =
\xi_0)\\
&&\qquad \leq\ep
\end{eqnarray*}
for all $a_{11} \in(- \rho, \rho)^3$.
This completes the proof.
\end{pf}

To deduce survival of strategy~1 from Lemma~\ref{lem:richardson-block}
and under the assumptions of the lemma, we now declare a site $(z, n)
\in H$ to be \emph{occupied} whenever
\[
x \in\eta_t \qquad\mbox{for all } (x, t) \in(Kz, cnK) + B_K
\times(0, cK)
\]
and define the set $\mathbb X_n$ of occupied sites at level $n$ as in
\eqref{eq:level}.
Repeating the proof of Lemma~\ref{lem:coex-perco} but using Lemma~\ref
{lem:richardson-block} in place of Lemmas~\ref
{lem:coex-target-1}--\ref
{lem:coex-target-3} directly gives the following result.

%
\begin{lemma}
\label{lem:richardson-perco}
Let $\ep> 0$ and $a_{11} = 1$.
Then there exist $K, c, \rho> 0$ and a coupling of the spatial game
with one dependent oriented site percolation such that
\[
\mathbb W_n^{1 - \ep} \subset\mathbb X_n\qquad \mbox{for
all } n \mbox{ whenever } \mathbb X_0 = \mathbb W_0^{1 - \ep}
\mbox{ and } \bar a_{11} \in(- \rho, \rho)^3.
\]
\end{lemma}

Taking $\ep> 0$ strictly smaller than one minus the critical value of
one dependent oriented site percolation, and using the coupling given
in the previous lemma for this value of $\ep$ as well
as the monotone function $f$ defined in \eqref{eq:coupling}, we obtain
that, for all $(x, t) \in\Z^d \times\R_+$
\begin{eqnarray*}&& P (x \in\eta_s \mbox{ for some } s > t)
\\
&&\qquad\geq P \bigl(\card(n \dvtx z \in\mathbb X_n) = \infty\mbox{ for all }
z \in \Z^d\bigr) = E f \bigl(\{\mathbb X_n \dvtx n \geq0
\}\bigr)
\\
&&\qquad\geq E f \bigl(\bigl\{\mathbb W_n^{1 - \ep} \dvtx n \geq0 \bigr
\}\bigr) = P \bigl(\card\bigl(n \dvtx z \in\mathbb W_n^{1 - \ep}
\bigr) = \infty\mbox{ for all } z \in\Z^d\bigr) = 1.
\end{eqnarray*}
This proves that strategy~1 survives but not that it wins since there
is a positive density of closed sites, which does not exclude the
possibility of having a positive density
of sites which are not occupied, and so the presence of type~2
players, at arbitrarily large times.
To prove extinction of the type~2 players, we use the coupling above
together with an idea of Lanchier \cite{lanchier_2012} that extends a
result of Durrett \cite{durrett_1992}.
This is done in the next lemma.

%
\begin{lemma}
\label{lem:richardson-dry}
Let $a_{11} = 1$.
Then, there exists $\rho> 0$ small such that
\[
\lim_{t \to\infty} P (x \notin
\eta_t) = 0 \qquad\mbox{for all } x \in\Z^d \mbox{ whenever } \bar
a_{11} \in(- \rho, \rho)^3.
\]
\end{lemma}

\begin{pf}
Following an idea of Lanchier \cite{lanchier_2012} we introduce the
new oriented graph $\mathcal H_2$ with the same vertex set as the
oriented graph $\mathcal H_1$ but in which there is an oriented edge
\begin{eqnarray} (z, n) \to\bigl(z', n'
\bigr) \nonumber
\\
\eqntext{\mbox{if and only if } \bigl(z' = z \pm e_i \mbox{ for some } i = 1, 2,
\ldots, d \mbox{ and } n' = n + 1\bigr)}\\
\eqntext{\mbox{or }
\bigl(z' = z \pm2 e_i \mbox{ for some } i = 1, 2, \ldots,
d \mbox { and } n' = n\bigr).}
\end{eqnarray}
See the right-hand side of Figure~\ref{fig:graphs} for a picture in $d
= 1$.
%
\begin{figure}

\includegraphics{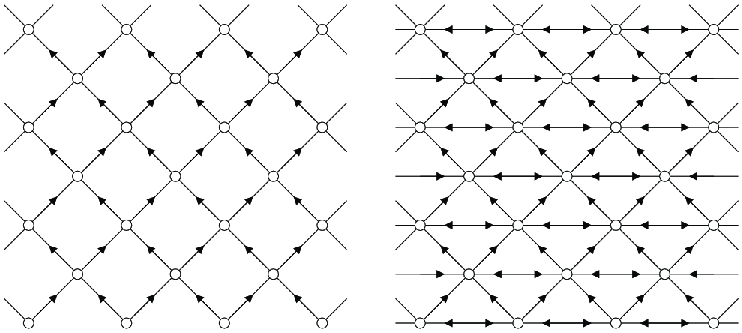}

\caption{Picture of the graphs $\mathcal H_1$ and $\mathcal
H_2$ in dimension $d = 1$.}
\label{fig:graphs}
\end{figure}
We say that a site is dry if it is not wet for oriented site
percolation on the graph $\mathcal H_1$.
Also, we write
\[
(w, 0) \to_j (z, n) \qquad\mbox{for } j = 1, 2,
\]
and say that there is a dry path in $\mathcal H_j$ connecting both
sites if there exist
\[
(z_0, 0) = (w, 0), (z_1, n_1), \ldots,
(z_{k - 1}, n_{k - 1}), (z_k, n_k) =
(z, n) \in H
\]
such that the following two conditions hold:
\begin{longlist}[(1)]
\item[(1)] $(z_i, n_i) \to(z_{i + 1}, n_{i + 1})$ is an oriented edge
in $\mathcal H_j$ for all $i = 0, 1, \ldots, k - 1$, and
\item[(2)] the site $(z_i, n_i)$ is dry for all $i = 0, 1, \ldots, k$.
\end{longlist}
Note that a dry path in $\mathcal H_1$ is also a dry path in $\mathcal
H_2$, but the reciprocal is false since the latter has more oriented
edges than the former.
The key to the proof is the following result:
if sites are closed with probability $\ep> 0$ sufficiently small, then
%
\begin{equation}
\label{eq:dry-1} %
 \lim_{n \to\infty} P
\bigl((w, 0) \to_2 (z, n) \mbox{ for some } w \in\Z^d\bigr)
= 0.
\end{equation}
In other words, if the density of open sites is close enough to one
then dry sites do not percolate even with the additional edges in
$\mathcal H_2$.
The proof for dry paths in the graph $\mathcal H_1$ directly follows
from Lemmas 4--11 in Durrett \cite{durrett_1992} but as pointed out in
\cite{lanchier_2012},
the proof easily extends to give the analog for dry paths in the
oriented graph $\mathcal H_2$.
To complete the proof, the last step is to show the connection between
dry paths and occupied sites.
Assume that
%
\begin{equation}
\label{eq:event-1} x \notin\eta_t \qquad\mbox{for some } (x, t) \in(Kz, cnK) +
B_K \times(0, cK),
\end{equation}
where $(z, n) \in H$.
Since a type~1 player can only change her strategy if there is a
type~2 player in her neighborhood, the event in \eqref{eq:event-1}
implies the existence of
\[
x_0, x_1, \ldots, x_m = x \in
\Z^d \quad\mbox{and}\quad s_0 = 0 < s_1 < \cdots<
s_{m + 1} = t
\]
such that the following two conditions hold:
\begin{longlist}[(1)]
\item[(1)] for all $j = 0, 1, \ldots, m$, we have $x_j \notin\eta_s$
for $s \in[s_j, s_{j + 1}]$, and
\item[(2)] for all $j = 0, 1, \ldots, m - 1$, we have $x_{j + 1} \in N_{x_j}$.
\end{longlist}
In particular, the spatial game being coupled with one dependent
oriented site percolation as in Lemma~\ref{lem:richardson-perco}, the event
in \eqref{eq:event-1} implies that there exists a dry path
%
\begin{equation}
\label{eq:event-2} (w, 0) \to_2 (z, n) \qquad\mbox{for some } w \in
\Z^d
\end{equation}
provided that the range $M \leq K$.
Note, however, that this does not imply the existence of a dry path in
the graph $\mathcal H_1$, which is the reason why we introduced a new
graph with additional edges.
In conclusion, the event in \eqref{eq:event-1} is a subset of the
event \eqref{eq:event-2}.
Since in addition $\ep$ can be made arbitrarily small by choosing the
parameter $K$ sufficiently large according to Lemma~\ref
{lem:richardson-perco}, taking the
probability of the two events above and using \eqref{eq:dry-1} implies that
\[
\lim_{t \to\infty} P (x \notin\eta_t) \leq \lim
_{n \to\infty} P \bigl((w, 0) \to_2 (z, n) \mbox{ for some } w
\in\Z^d\bigr) = 0
\]
for all $x \in\Z^d$ where $(z, n)$ is as in \eqref{eq:event-1}.
This completes the proof.
\end{pf}

To complete the proof of the theorem, we let $\rho> 0$ as in the
previous lemma.
Then, having arbitrary payoffs $a_{12}$ and $a_{21}$ and $a_{22}$,
there exists $m$ large such that $\bar a_{11} \in(- m \rho, m \rho)$.
Since in addition the limiting behavior of the spatial game remains
unchanged by speeding up time, Lemma~\ref{lem:richardson-dry} implies
that strategy~1 wins for all $a_{11} > m$.


\section{Proof of Theorem~\texorpdfstring{\protect\ref{th:walks}}{3}}
\label{sec:walks}

This section is devoted to the proof of Theorem~\ref{th:walks}.
The intuition behind this result is simple, though the arguments to
make the proof rigorous are somewhat more challenging.
To understand the theorem heuristically, observe that, in the limiting
case $a_{11} = - \infty$ and provided one starts from a
suitable initial configuration, the process becomes instantaneously
sparse: configurations where two type~1 players are
neighbors are not possible.
Since in addition type~2 players can only change their strategy when
they are located in the neighborhood of at least one type~1 player,
the process is dominated by a system of annihilating particles:
as long as several particles are in the same interaction neighborhood,
one of them is instantaneously killed.
In particular, the density of type~1 players can only decrease.
Under the assumption $a_{12} < 0$, these particles also die
spontaneously, which implies that the density of type~1 players
decreases to zero.

The main difficulty to prove the theorem is to extend this
heuristic argument to the nonlimiting case when the payoff $a_{11}$ is
small but different from $- \infty$.
We start with some key definitions and a brief overview of the global
strategy of our proof.
Identifying again configurations with the set of type~1 players, we
say that a set/configuration $\eta$ is \emph{sparse} whenever
\[
x, y \in\eta\qquad\mbox{implies that } y \notin N_x.
\]
We also say that configuration $\eta$ is sparse in $B$ if the set
$\eta
\cap B$ is sparse.
For the spatial game, we say that there is a \emph{type~\textup{1} invasion
path} $(x, r) \leadsto(y, t)$ if there are
\[
x_0 = x, x_1, \ldots, x_n = y \in
\Z^d \quad\mbox{and}\quad s_0 = r < s_1 < \cdots<
s_n < s_{n + 1} = t \in\R_+
\]
such that the following three conditions hold:
\begin{itemize}
\item for $i = 0, 1, \ldots, n$, we have $x_i \in\eta_s$ for all $s_i
\leq s \leq s_{i + 1}$;
\item for $i = 1, 2, \ldots, n$, we have $x_i \notin\eta_{s_i-}$;

\item for $i = 1, 2, \ldots, n$, we have $x_i \in N_{x_{i - 1}}$.
\end{itemize}
Note that there exists a type~1 invasion path $(x, 0) \leadsto(y, t)$
if and only if the player at site $y$ at time $t$ follows strategy~1
since type~2 players can only change their strategy if they are in the
neighborhood of a type~1 player.
Finally, we call an invasion path:
\begin{itemize}
\item an \emph{inner path} whenever $x_i \in B_{4K - M}$ for all $i
= 0,
1, \ldots, n$,
\item an \emph{outer path} whenever $x_i \notin B_{4K}$ for all $i = 0,
1, \ldots, n$,
\item a \emph{transversal path} whenever $x_i \in B_{4K} \setminus B_{4K
- M}$ for some $i = 0, 1, \ldots, n$.
\end{itemize}
To prove extinction of strategy~1 using a block construction, the main
ingredient is to prove that if the region $B_K$ is empty initially,
then the region $B_{2K}$ will, with probability close to one for
suitable parameters, be empty at a later time that we choose to be $2
\sqrt K$.
To show this result, we observe that, since type~1 players are located
on type~1 invasion paths, it suffices to prove that:
\begin{longlist}[(1)]
\item[(1)] the probability that an inner path lasts more than $2 \sqrt
K$ units of time is small, and
\item[(2)]the probability that a transversal path reaches $B_{2K}$ by
time $2 \sqrt K$ is small.
\end{longlist}
Note that outer paths are unimportant in proving the theorem because,
by definition, they do not reach the target region.
The proof of the second assertion simply relies on the fact that, with
probability close to one when $K$ is large, and regardless of the value
of the payoffs,
invasion paths expand at most linearly.
The proof of the first assertion is more involved and is divided into
three steps.
First, we show that the process is sparse in $B_{4K}$ by time 1, then
that it is sparse a positive fraction of time in this box until
time $2 \sqrt K$, and finally that conditional on this previous event,
inner paths die out exponentially fast.
Throughout this section:
\begin{itemize}
\item$\mathcal S$ denotes the set of sparse configurations,
\item$\Omega_{\mathrm{in}}$ is the event that there is an inner path from time
0 to time $2 \sqrt K$ and
\item$\Omega_{\mathrm{tr}}$ is the event that a transversal path reaches
$B_{2K}$ by time $2 \sqrt K$.
\end{itemize}
To begin with, we prove that if the region $B_K$ is initially void in
type~1 players, then the process becomes sparse in $B_{4K}$ after a
short time.

%
\begin{lemma}
\label{lem:walks-sparse}
Let $\ep> 0$.
For all $K$, there exists $m_1:= m_1 (\ep, K, \bar a_{11}) < \infty$
such that
\begin{eqnarray}
P \bigl(\eta_s \cap B_{4K} \in\mathcal S \mbox{ for some }
s \in(0, 1) | \eta_0 \cap B_K = \varnothing\bigr) \geq1
- \ep/ 3\nonumber \\
\eqntext{\mbox {for all } a_{11} < - m_1.}
\end{eqnarray}
\end{lemma}

\begin{pf}
To begin with, we observe that when $a_{11} = -1$ and $\bar a_{11} = 0$:
\begin{itemize}
\item type~1 players with at least one type~1 neighbor and at least one
type~2 neighbor change their strategy at a positive rate whereas

\item type~2 players all have a zero payoff, so they do not change
their strategy.
\end{itemize}
This implies that there exists $a:= a (\ep) > 0$ such that, for all $K
> 0$,
\[
P (\eta_{aK} \cap B_{4K} \notin\mathcal S |
\eta_0 \cap B_K = \varnothing) \leq\ep/ 6 \qquad\mbox{when }
\bar a_{11} = 0,
\]
which in turn implies the existence of $\rho:= \rho(K, \ep)$ such that
%
\begin{equation}
\label{eq:sparse} P (\eta_{aK} \cap B_{4K} \notin\mathcal S |
\eta_0 \cap B_K = \varnothing) \leq\ep/ 3 \qquad\mbox{for
all } \bar a_{11} \in(- \rho, \rho)^3.
\end{equation}
For arbitrary $\bar a_{11}$, we fix $m_1:= m_1 (\bar a_{11}, \rho)$
such that
\[
\bar a_{11} \in(- m_1 \rho, m_1
\rho)^3 \quad\mbox{and}\quad m_1 > aK.
\]
Then, \eqref{eq:sparse} directly implies that
\[
P (\eta_s \cap B_{4K} \notin\mathcal S |
\eta_0 \cap B_K = \varnothing) \leq\ep/ 3
\]
for all $a_{11} < - m_1$ and for $s = aK / m_1 < 1$, which proves the lemma.
\end{pf}

Before proving that inner paths die out exponentially fast, we need an
additional preliminary result that
ensures that the configuration in~$B_{4K}$ is almost sparse for a
large amount of time when~$a_{11}$ is small.
The proof slightly differs depending on the sign of the two
payoffs~$a_{22}$ and~$a_{21}$.
Since the proof when these two payoffs are negative is more difficult
and requires additional arguments, we only focus on this case.
Under this assumption and the assumptions of the theorem, all four
payoffs are negative, in which case all the players have a positive death
rate and the graphical representation of the process can be
reformulated in the following manner.
We introduce the following collections of independent random variables:
for all $x \in\Z^d$ and $i, j = 1, 2$ and $n > 0$:
\begin{itemize}
\item we let $T_n (x, i, j) =$ the $n$th arrival time of a Poisson
process with rate $- a_{ij}$;
\item we let $U_n (x, i, j) =$ a uniform random variable over the
interaction neighborhood $N_x$;
\item we let $V_n (x, i, j) =$ a uniform random variable over the
interaction neighborhood $N_x$.
\end{itemize}
At the arrival times $T_n (x, i, j)$, we draw
%
\begin{equation}
\label{eq:walks-arrow} \mbox{an arrow } V_n (x, i, j) \to x \qquad\mbox{with the
label } \bigl(U_n (x, i, j), i, j\bigr)
\end{equation}
and say that this arrow is \emph{active} whenever
\[
\eta_{t-} (x) = i \quad\mbox{and}\quad \eta_{t-} \bigl(U_n
(x, i, j)\bigr) = j \qquad\mbox{where } t:= T_n (x, i, j).
\]
Given an initial configuration and a realization of this graphical
representation, the process can be constructed going forward in time by setting
\begin{eqnarray*} \eta_t (x) &:= &
\eta_{t-} \bigl(V_n (x, i, j)\bigr) \qquad \mbox{if the arrow in
\eqref{eq:walks-arrow} is active},
\\
\eta_t (x) &:= & \eta_{t-} (x) \qquad \mbox{if the arrow in
\eqref {eq:walks-arrow} is not active},
\end{eqnarray*}
where again $t:= T_n (x, i, j)$.
For any given $K > 0$, we let $\tau_0 = 0$ and
\begin{eqnarray*} \tau_i &:= & \inf\bigl
\{T_n (x, 2, j) > \tau_{i - 1} \mbox{ for some } n > 0 \mbox{ and }
x \in B_{4K} \mbox{ and } j = 1, 2 \bigr\},
\\
\rho_i &:= & (1/2) (\tau_i + \tau_{i + 1})\qquad
\mbox{for } i \geq1
\end{eqnarray*}
and say that the arrow at time $\tau_i$ is \emph{good} whenever
\begin{eqnarray*}&&\mbox{there is at least one $(x, 1, 1)$-arrow
$z' \to z$} \\
&&\qquad\mbox{for all $z \in N_x$ and all
$z' \in N_z$ between time $\tau_i$ and
time $\rho_i$},
\end{eqnarray*}
where vertex $x$ is the vertex fixed while defining $\tau_i$.
Note that if the arrow at time $\tau_i$ is good, and the configuration
in $B_{4K}$ is sparse just before $\tau_i$, and the player at $x$
becomes of type~1
at time $\tau_i$, then all the type~1 players in the neighborhood of
$x$ become of type~2 by time $\rho_i$ unless the player at $x$ changes
her strategy before.
In either case, the configuration will be sparse in $B_{4K}$ between
the two times $\rho_i$ and $\tau_{i + 1}$.
Define the stopping time
\[
\sigma_K:= \inf\{\tau_i \dvtx \mbox{ the arrow at time
$\tau_i$ is not good} \}.
\]
Then we have the following lemma.

%
\begin{lemma}
\label{lem:walks-good}
Let $\ep> 0$. For all $K$, there exists $m_2:= m_2 (\ep, K, \bar
a_{11}) < \infty$ such that
\[
P (\sigma_K < 2 \sqrt K) \leq\ep/ 6 \qquad\mbox{for all } a_{11}
< - m_2.
\]
\end{lemma}

\begin{pf}
First, we observe that
\[
J:= \sup\{j \dvtx \tau_j < 2 \sqrt K \} = \poisson\bigl(- 2
(a_{21} + a_{22}) \sqrt K (8K + 1)^d\bigr)
\]
and fix $m:= m (\ep, K, \bar a_{11})$ such that
%
\begin{equation}
\label{eq:good-1} P (J > m) \leq\ep/ 18.
\end{equation}
Letting $\rho:= \rho(\ep, K, \bar a_{11}) = - \ep(36 m (a_{21}
+ a_{22})(8K + 1)^d)^{-1}$ and using that
\[
\tau_{i + 1} - \tau_i = \exponential\bigl(-
(a_{21} + a_{22}) (8K + 1)^d\bigr),
\]
we also have
%
\begin{eqnarray}
\label{eq:good-2} %
&&P (\tau_{i + 1} -
\tau_i < 2 \rho\mbox{ for some } i = 0, 1, \ldots, J - 1 | J \leq m)
\nonumber\\
&&\qquad\leq P (\tau_{i + 1} - \tau_i < 2 \rho\mbox{ for some } i = 0,
1, \ldots, m - 1)
\nonumber
\\[-8pt]
\\[-8pt]
\nonumber
&&\qquad\leq m \bigl(1 - \exp\bigl(2 (a_{21} + a_{22}) (8K +
1)^d \rho\bigr)\bigr)
\\
&&\qquad\leq- 2 m (a_{21} + a_{22}) (8K + 1)^d \rho=
\ep/ 18. \nonumber
\end{eqnarray}
Finally, since for all $x \in B_{4K}$, all $z \in N_x$ and all $z' \in N_z$,
\begin{eqnarray*} &&P \bigl(\mbox{there is no $(x, 1, 1)$-arrow
$z' \to z$ between time $\tau _i$ and time $
\tau_i + \rho$}\bigr)\\
&&\qquad = P \bigl(\exponential\bigl(-
a_{11} (2M + 1)^{-2d}\bigr) > \rho\bigr) = \exp
\bigl(a_{11} (2M + 1)^{-2d} \rho\bigr)
\end{eqnarray*}
defining $m_2:= m_2 (\ep, K, \bar a_{11})$ by
\[
m_2:= - (2M + 1)^{2d} \rho^{-1} \ln\bigl(\ep
\bigl(18 m (2M + 1)^{2d}\bigr)^{-1}\bigr) > 0,
\]
we obtain the conditional probability
%
\begin{eqnarray}
\label{eq:good-3} %
&&P (\sigma_K < \sqrt K
| J \leq m \mbox{ and } \tau_{i + 1} - \tau_i > 2 \rho\mbox{ for
all } i = 0, 1, \ldots, J - 1) \nonumber\\
&&\qquad\leq m P (\mbox{the arrow at time }
\tau_i \mbox{ is not good } | \tau_{i + 1} - \tau_i
> 2 \rho)\\
&&\qquad \leq m (2M + 1)^{2d} \exp\bigl(a_{11} (2M +
1)^{- 2d} \rho\bigr) \leq \ep/ 18\nonumber
\end{eqnarray}
for all $a_{11} < - m_2$.
The result follows by observing that the probability to be estimated
is smaller than the sum of the three probabilities in \eqref
{eq:good-1}--\eqref{eq:good-3}.
\end{pf}

With Lemma~\ref{lem:walks-good} in hand, we are now ready to prove
that the conditional probability given that the initial configuration
is sparse, that an inner path
lasts more than $2 \sqrt K$ units of time is small when the scaling
parameter $K$ is large, and the payoff $a_{11}$ is small.

%
\begin{lemma}
\label{lem:walks-vertical}
Let $\ep> 0$ and $a_{12} < 0$. Then
\[
P (\Omega_{\mathrm{in}} | \eta_0 \cap B_{4k} \in
\mathcal S) \leq\ep / 3\qquad  \mbox{for all $K$ large and } a_{11} < -
m_2.
\]
\end{lemma}

\begin{pf}
Recall that $\rho_j:= (1/2)(\tau_j + \tau_{j + 1})$ and introduce the
set-valued process
\[
Q_t:= \bigl\{y \in B_{4K} \dvtx \mbox{ there is an inner
path } (x, 0) \leadsto(y, t) \mbox{ for some } x \in B_{4K} \bigr\}.
\]
As pointed out above, if $\eta_0 \cap B_{4K}$ is sparse and $\sigma_K >
2 \sqrt K$, then:
\begin{itemize}
\item$Q_t \subset\eta_t \cap B_{4K}$ is sparse for all $t \in(\rho
_j, \tau_{j + 1})$, $j = 0, 1, \ldots, J - 1$, and
\item$\card Q_{\rho_j} \leq\card Q_{\tau_j}$ for all $j = 1, 2,
\ldots, J$.
\end{itemize}
Since in addition type~1 players with only type~2 neighbors (which is
the case for all type~1 players in sparse configurations) change their
strategy at rate $- a_{12}$,
\[
 \lim_{h \to0} h^{-1} P
\bigl(Q_{t + h} = Q_t - \{y \} \bigr) = - a_{12}\qquad
\mbox{for all } t \in(\rho_j, \tau_{j + 1}) \mbox{ and } y \in
Q_t.
\]
Using also that $Q_t$ is sparse for at least $\sqrt K$ time units
before time $2 \sqrt K$, we obtain
\begin{eqnarray*} &&P (\Omega_{\mathrm{in}} | \eta_0
\cap B_{4K} \in\mathcal S \mbox{ and } \sigma_K > 2 \sqrt K)\\
&&\qquad = P \bigl(Q_t \neq \varnothing\mbox{ for all } t \in(0, 2 \sqrt K) |
\eta_0 \cap B_{4K} \in\mathcal S \mbox{ and }
\sigma_K > 2 \sqrt K\bigr)
\\
&&\qquad\leq(\card B_{4K}) P \bigl(\exponential(- a_{12}) > \sqrt K
\bigr) \\
&&\qquad= (\card B_{4K}) \exp(a_{12} \sqrt K) \leq \ep/ 6
\end{eqnarray*}
for all $K$ large enough.
In particular, for all such $K$ and all $a_{11} < - m_2 (\ep, K, \bar a_{11})$,
\begin{eqnarray*} P (\Omega_{\mathrm{in}} | \eta_0
\cap B_{4K} \in\mathcal S) & \leq& P (\Omega_{\mathrm{in}} |
\eta_0 \cap B_{4K} \in\mathcal S \mbox{ and }
\sigma_K > 2 \sqrt K)+ P (\sigma _K < 2 \sqrt K)
\\
&\leq&\ep/ 6 + \ep/ 6 = \ep/ 3
\end{eqnarray*}
according to Lemma~\ref{lem:walks-sparse}.
This completes the proof.
\end{pf}

The next lemma shows the analog for transversal paths:
with probability close to one, none of the transversal paths reaches
the target region $B_{2K}$ by time $2 \sqrt K$.

%
\begin{lemma}
\label{lem:walks-horizontal}
Let $\ep> 0$ and $a_{12} < 0$. For all $K$ large, $P (\Omega_{\mathrm{tr}})
\leq\ep/ 3$.
\end{lemma}

\begin{pf}
We introduce the rates
\begin{eqnarray*} \mu_K &:= & - 2 \sqrt K
(a_{21} + a_{22}) \card(B_{4K} \setminus
B_{4K - M}) > 0,
\\
\nu_K &:= & - 2 \sqrt K (a_{21} + a_{22}) > 0.
\end{eqnarray*}
First, we observe that
\begin{eqnarray*} m_K &:= & \mbox{the number of $(x,
2, j)$-arrows}
\\
& & \mbox{that point at the region } (B_{4K} \setminus
B_{4K - M}) \times(0, 2 \sqrt K)\\ &=& \poisson(\mu_K)
\end{eqnarray*}
from which it follows that
%
\begin{equation}
\label{eq:horizontal-2} P (m_K > 2 \mu_K) \leq\ep/ 6 \qquad\mbox{for
all $K$ sufficiently large}.
\end{equation}
Now, let $n_l$ be the number of type~1 invasion paths of length $l$
%
\begin{eqnarray}
\label{eq:horizontal-3} (x, t) \leadsto(y, 2 \sqrt K)
\nonumber
\\[-8pt]
\\[-8pt]
\eqntext{\mbox{starting from some } (x, t)
\in(B_{4K} \setminus B_{4K - M}) \times(0, 2 \sqrt K),}
\end{eqnarray}
and observe that, on the event that $m_K \leq2 \mu_K$, we have
%
\begin{eqnarray}
\label{eq:horizontal-4} n_l &\leq& 2 \mu_K (2M + 1)^{ld}
\nonumber
\\[-8pt]
\\[-8pt]
\nonumber
&=& - 4 \sqrt K (a_{21} + a_{22}) (2M + 1)^{ld}
\card(B_{4K} \setminus B_{4K - M}).
\end{eqnarray}
In addition, if in \eqref{eq:horizontal-3} site $y \in B_{2K}$, then
the length must be at least
\[
l \geq(2K - M) / M \geq K / M.
\]
Also, since each type~2 player changes her strategy at rate at most $-
(a_{21} + a_{22})$, the probability of any given type~1 invasion path
\eqref{eq:horizontal-3} of length at least $l \geq K / M$ is bounded by
%
\begin{eqnarray}
\label{eq:horizontal-5} %
P \bigl(\poisson(
\nu_K) \geq l\bigr) & = & \sum_{n = l}^{\infty}
\bigl(\nu_K^n / n!\bigr) e^{-\nu_K} \leq2 \bigl(
\nu _K^l/ l!\bigr) e^{- \nu_K}
\nonumber\\
& \leq& \bigl(4 e^{- \nu_K} / \sqrt{2 \pi l}\bigr) \bigl(l^{-1} e
\nu_K\bigr)^l \leq\bigl(l^{-1} e
\nu_K\bigr)^l
\nonumber
\\[-8pt]
\\[-8pt]
\nonumber
&\leq&\bigl(K^{-1} eM
\nu_K\bigr)^l
\\
& = & \bigl(- 2 (a_{21} + a_{22}) eM / \sqrt K
\bigr)^l \nonumber
\end{eqnarray}
for all $K$ sufficiently large where the second inequality follows from
Stirling's formula.
To complete the proof of the lemma, we combine \eqref
{eq:horizontal-2}, \eqref{eq:horizontal-4} and \eqref{eq:horizontal-5}
to obtain
\begin{eqnarray*} P (\Omega_{\mathrm{tr}}) & \leq& P
(m_K > 2 \mu_K) + P (\Omega_{\mathrm{tr}} |
m_K \leq2 \mu_K)
\\
& \leq& \ep/ 6 + \sum_{l = K/M}^{\infty} 2
\mu_K \bigl((2M + 1)^d \bigl(- 2 (a_{21} +
a_{22}) eM / \sqrt K\bigr)\bigr)^l \leq\ep/ 3
\end{eqnarray*}
for all $K$ sufficiently large.
\end{pf}

Having proved Lemmas~\ref{lem:walks-sparse}--\ref
{lem:walks-horizontal}, we are now ready to compare the process with
oriented site percolation and deduce
almost sure extinction of strategy~1.
The final part of the proof follows from the same arguments as for
Lemma~\ref{lem:richardson-dry}.
We say that a site $(z, n) \in H$ is \emph{void} whenever
\[
\eta_{2n \sqrt K} \cap(Kz + B_K) = \varnothing
\]
and define the set $\mathbb X_n$ of void sites at level $n$ by
\[
\mathbb X_n:= \bigl\{z \in\Z^d \dvtx (z, n) \in H
\mbox{ and is void} \bigr\}.
\]
The coupling with oriented site percolation is given in the next lemma.

%
\begin{lemma}
\label{lem:walks-perco}
Let $\ep> 0$ and $a_{12} < 0$.
Then there exist $K$ large and a coupling of the spatial game with
four dependent oriented site percolation such that
\[
\mathbb W_n^{1 - \ep} \subset\mathbb X_n \qquad\mbox{for
all } n \mbox{ whenever } \mathbb X_0 = \mathbb W_0^{1 - \ep}
\mbox { and } a_{11} < - \max(m_1, m_2).
\]
\end{lemma}

\begin{pf}
Combining Lemmas~\ref{lem:walks-sparse},~\ref{lem:walks-vertical}
and~\ref{lem:walks-horizontal}, we obtain
\begin{eqnarray*}&& P (\eta_{2 \sqrt K} \cap B_{2K}
\neq\varnothing| \eta_0 \cap B_K = \varnothing) \\
&&\qquad\leq P (
\Omega_{\mathrm{in}} | \eta_0 \cap B_K = \varnothing)
+ P (\Omega_{\mathrm{tr}} | \eta_0 \cap B_K =
\varnothing)\\
&&\qquad \leq P \bigl(\eta_s \cap B_{4K} \notin
\mathcal S \mbox{ for all } s \in(0, 1) | \eta_0 \cap B_K
= \varnothing\bigr)
\\
&&\qquad\quad{}+ P \bigl(\Omega_{\mathrm{in}} | \eta_s \cap B_{4K}
\in\mathcal S \mbox{ for some } s \in(0, 1)\bigr) + P (\Omega_{\mathrm{tr}} |
\eta_0 \cap B_K = \varnothing )\\
&&\qquad \leq\ep/ 3 + \ep/ 3 +
\ep/ 3 = \ep
\end{eqnarray*}
for all $K$ large and all $a_{11} < - \max(m_1, m_2)$.
The existence of a coupling with oriented site percolation then
follows from the same arguments as in the proof of Lemma~\ref{lem:coex-perco}.
Note that the comparison can only be made with four dependent
percolation because all the events introduced in the proofs of
Lemmas~\ref{lem:walks-sparse}--\ref{lem:walks-horizontal} are
measurable with respect to the graphical representation in $B_{4K}$.
\end{pf}

Repeating the same steps as in the previous two sections, we deduce
from the lemma that strategy~2 survives.
To also prove extinction of strategy~1, we observe that Lemma~\ref
{lem:walks-horizontal} excludes the existence of transversal paths that
ever intersect $B_{2K}$
by time $2 \sqrt K$ with probability close to one.
In particular, including this event in our definition of void sites,
Lemma~\ref{lem:walks-perco} still holds for arbitrarily small $\ep$.
Moreover, with this new definition, we obtain that
\[
\eta_{2n \sqrt K} \cap(Kz + B_K) \neq\varnothing\qquad\mbox {implies }
(w, 0) \to_2 (z, n) \mbox{ for some } w \in\Z^d
\]
in the oriented graph $\mathcal H_2$.
Note that this implication is similar to the one in the proof of
Lemma~\ref{lem:richardson-dry} and can be shown using the same idea.
Almost sure extinction of strategy~1 then follows from the lack of
percolation of the dry sites repeating again the steps in the proof of
Lemma~\ref{lem:richardson-dry}.


\section*{Acknowledgments}
The author would like to thank Rick Durrett for pointing out the
important literature on evolutionary game theory as well as an
anonymous referee whose comments helped improve the clarity of some proofs.


%



\printaddresses

\end{document}